\documentclass[12pt]{amsart}
\usepackage{graphics,wrapfig, graphicx, mathrsfs,xcolor}
\usepackage{mathtools}
\usepackage{amssymb,amscd,graphicx,xcolor,subfig,tikz}
\usepackage{graphics}
\usepackage{indentfirst}
\usepackage{bm, enumerate,xcolor}
\usepackage[dvips]{epsfig}
\usepackage{latexsym}
\usepackage{float}

\usepackage{graphics,url,color,epsfig}
\usepackage{tikz}
\usepackage{xcolor}
\definecolor{mysoftblue}{hsb}{0.55,0.15,0.9}
\definecolor{mysoftpurple}{hsb}{0.9,0.25,0.9}
\definecolor{mysoftorange}{hsb}{0.15,0.3,0.95}
\usetikzlibrary{patterns}

\usepackage[colorlinks]{hyperref}
\AtBeginDocument{
   \hypersetup{
    linkcolor=blue,
    citecolor=blue,
 }
}

\usepackage{amsmath}
\usepackage[applemac]{inputenc}
\usepackage[T1]{fontenc}
\newtheorem{lemma}{Lemma}[section]
\newtheorem{theorem}{Theorem}[section]
\newtheorem{corollary}{Corollary}[section]

\newtheorem{definition}{Definition}[section]

\numberwithin{equation}{section} \numberwithin{theorem}{section}
\numberwithin{example}{section} \numberwithin{remark}{section}
\numberwithin{figure}{section} \numberwithin{algorithm}{section}
\mathtoolsset{showonlyrefs}

\begin{document}
\title[Phase transitions in two-component Bose-Einstein condensates (II)]{Phase transitions in two-component Bose-Einstein condensates with Rabi frequency (II): The De Giorgi conjecture for the nonlocal problem in $\mathbb{R}^{2}$ or $\mathbb{R}^{3}$}
\author{Leyun Wu}
\address{School of Mathematics, South China University of Technology, Guangzhou, 510640, P. R. China}\email{leyunwu@scut.edu.cn}
\author{Chilin Zhang}
\address{School of Mathematical Sciences, Fudan University, Shanghai 200433, P. R. China}\email{zhangchilin@fudan.edu.cn}

\begin{abstract}
    In this series of papers, we investigate coupled systems arising in the study of two-component Bose-Einstein condensates, and we establish classification results for solutions of De Giorgi conjecture type.
    
    In the present (second) paper of the series, we focus on the nonlocal problem of the form
    \begin{equation*}
        \left\{\begin{aligned}
        (-\Delta)^{s}u+u(u^{2}+v^{2}-1)+v(\alpha uv-\omega)=0,\\
        (-\Delta)^{s}v+v(u^{2}+v^{2}-1)+u(\alpha uv-\omega)=0,
    \end{aligned}
    \right.
    \end{equation*}
    which models the stationary states of Rabi-coupled condensates with inter- and intra-species interactions. 
    
We prove that for $\frac{1}{2}\le s<1$, any positive entire solution $(u,v)$ in $\mathbb{R}^3$ satisfying the monotonicity condition $\partial_{x_3}u>0>\partial_{x_3}v$ must be one-dimensional. Moreover, when $0<s<\frac{1}{2}$, the same conclusion holds for monotone solutions in $\mathbb{R}^2$.

Our work generalizes classical De Giorgi-type theorems to a new class of nonlocal coupled systems and, to the best of our knowledge, presents the first Liouville-type classification of monotone solutions for Rabi-coupled fractional Bose-Einstein condensates, with particular emphasis on fractional Gross-Pitaevskii models.

\end{abstract}

\maketitle
\\
2020 {\bf MSC.} 35Q56, 82B26, 35J47, 35B53.
\section{Introduction}

The celebrated \emph{De Giorgi conjecture}, originally proposed in \cite{DeGiorgi1979}, asserts that any global solution to
\begin{equation*}
\Delta u = u^{3} - u \quad \text{in } \mathbb{R}^{n},
\end{equation*}
that satisfies the monotonicity condition $\frac{\partial u}{\partial x_{n}} > 0$, is necessarily one-dimensional, at least in sufficiently low dimensions.
This conjecture was affirmatively resolved in dimensions $n \leq 3$ by \cite{GG98,AC00}, whereas a counterexample in $\mathbb{R}^{9}$ was later constructed in \cite{dPKW11}.
For the intermediate range $4 \leq n \leq 8$, the conjecture remains open; however, several partial results have been established under additional structural or symmetry hypotheses, see for instance, \cite{GG03,S09}.

\subsection{The Allen-Cahn equation and minimal surfaces}

The Allen--Cahn equation,
\begin{equation*}
    (-\Delta)^{s} u = u^{3} - u,
\end{equation*}
plays a central role in the study of phase transitions and serves as a prototype connecting reaction--diffusion equations with the theory of minimal surfaces. 
In our previous work \cite{WZ25}, we showed that in the special case $\alpha = 2$, the coupled Gross--Pitaevskii system
\begin{equation*}
    \left\{\begin{aligned}
        \Delta u=u(u^{2}+v^{2}-1)+v(\alpha uv-\omega),\\
        \Delta v=v(u^{2}+v^{2}-1)+u(\alpha uv-\omega),
    \end{aligned}
    \right.\quad u,v>0.
\end{equation*}
can be decoupled so that $u + v$ is constant, and the difference $u - v$ satisfies the single Allen--Cahn equation after a suitable rescaling. 
This observation strongly suggests that this system inherits structural properties analogous to those of the Allen--Cahn equation.

The Allen--Cahn equation can be interpreted as the Euler--Lagrange equation of the classical Ginzburg--Landau energy
\begin{equation*}
    J(u,\Omega) = \int_{\Omega} \Big\{ \frac{|\nabla u|^{2}}{2} + W(u) \Big\} \, dx,
    \qquad W(u) = \frac{1}{4}(1-u^{2})^{2}\chi_{[-1,1]}.
\end{equation*}
The constant functions $u \equiv \pm 1$ 
represent two stable equilibria, while the most interesting solutions are the phase-transition profiles that connect these states through a narrow interfacial region, commonly referred to as the \emph{transition layer}.
It is well known that energy-minimizing phase transition solutions are intimately connected to minimal surfaces. 
Indeed, considering the rescaled functions $u_{R}(y) = u(Ry)$, one obtains the rescaled Ginzburg-Landau functional
\begin{equation*}
    J_{R}(u_{R},\Omega) = \int_{\Omega} \Big\{ \frac{|\nabla u_{R}|^{2}}{2R} + R W(u_{R}) \Big\} \, dy.
\end{equation*}

Modica and Mortola \cite{M79,MM80} proved that as $R \to \infty$, the sequence of energies $J_{R}$ $\Gamma$-converges to the perimeter functional for Caccioppoli sets. 
Consequently, a subsequence of $u_{R}$ converges (in the weak $BV$ and strong $L^{1}_{\mathrm{loc}}$ sense) to
\begin{equation*}
    u_{\infty} = \chi_{E} - \chi_{E^{c}},
\end{equation*}
where $E$ is a Caccioppoli set of minimal perimeter.

What is equally important is that, as long as we assume $u(0)=0$, the boundary of the limiting Caccioppoli set $E$ necessarily passes through the origin. This follows from a density estimate due to Caffarelli and C\'ordoba~\cite{CC95}, which asserts that the volumes of the positive and negative sets of $u(x)$ are comparable. Combining the $\Gamma$-convergence result with this density estimate, and invoking the Bernstein theorem for minimal surfaces established by Simons~\cite{S68}, we conclude that the zero level set of a global energy minimizer $u(x)$ in $\mathbb{R}^{n}$ must be asymptotically flat, provided that $n \geq 7$.

For such a single equation, a number of important results have been established concerning the classification of global solutions. In \cite{GG98}, Ghoussoub and Gui proved De Giorgi's conjecture in $\mathbb{R}^{2}$, and later Ambrosio and Cabr\'e~\cite{AC00} confirmed it in $\mathbb{R}^{3}$. In dimensions $\mathbb{R}^{4}$ through $\mathbb{R}^{8}$, the validity of De Giorgi's conjecture remains open, except under certain additional assumptions. For instance, Ghoussoub and Gui~\cite{GG03} obtained the flatness of solutions under an additional odd-symmetry condition. Moreover, Savin~\cite{S09} proved De Giorgi's conjecture in $\mathbb{R}^{8}$, assuming that the zero level set is an entire graph. In fact, the asymptotic flatness of the level sets implies the flatness of the global solution itself.

De Giorgi's conjecture exhibits a dimensional limitation. In particular, del Pino, Kowalczyk, and Wei~\cite{dPKW11} constructed a non-trivial global solution in $\mathbb{R}^{9}$. Their construction is based on the Simons cone~\cite{S68}, the most well-known non-flat minimal graph, which occurs only in high dimensions.

\subsection{The fractional De Giorgi conjecture}

In recent years, considerable attention has been devoted to the fractional counterpart of De Giorgi's conjecture, in which the classical Laplacian in the Allen-Cahn equation is replaced by the fractional Laplacian. More precisely, one considers bounded entire solutions of
\begin{equation}\label{eq:fractional-AC}
    (-\Delta)^{s} u = u - u^{3} \quad \text{in } \mathbb{R}^{n}, 
    \qquad 0<s<1,
\end{equation}
satisfying the monotonicity condition
\[
    \frac{\partial u}{\partial x_{n}} > 0 \quad \text{in } \mathbb{R}^{n}.
\]
It is natural to ask a parallel question as in the local case, i.e.: the \emph{fractional De Giorgi conjecture}. In particular, one might conjecture that for small dimensions $n$'s, any such solution in $\mathbb{R}^{n}$ must be one-dimensional.

However, we do not yet have a definite dimensional obstruction when raising the conjecture for the nonlocal De Giorgi conjecture. The main reason lies in the $\Gamma$-convergence issue. In fact, it was shown in Savin-Valdinoci \cite{SV12} that the $\Gamma$-limit differs fundamentally between the case $s\in(0,\frac{1}{2})$ and the case $s\in[\frac{1}{2},1)$. In the latter case, the $\Gamma$-limit is the set characteristic function of a set with minimal perimeter, which is exactly the same limit as in the local case. Not surprisingly, it has been shown by Savin in \cite{S18a,S18b} that the fractional De Giorgi conjecture holds true in $\mathbb{R}^{8}$ in this case for $s\geq\frac{1}{2}$, given that the zero set is an entire graph. On the other hand, in the former case, the $\Gamma$-limit is the characteristic function of a set with fractional perimeter, or simply put, the nonlocal minimal surface that was previously systematically investigated in Caffarelli-Roquejoffre-Savin \cite{CRS10}. It was shown in \cite{CabreCinti2014CVPDE,CS14,DSV20} that the fractional De Giorgi conjecture holds in $\mathbb{R}^{2}$ when $s<\frac{1}{2}$. This result coincides with the classification result in \cite{SV13} for fractional minimal cones in $\mathbb{R}^{2}$.


On the other hand, the general case remains open for $4 \leq n \leq 8$ in the absence of additional hypotheses, while counterexamples are expected for $n \geq 9$, in analogy with the local theory. This line of research bridges the study of nonlocal operators with rigidity phenomena in elliptic PDEs, and has spurred deep investigations into fractional Allen-Cahn equations, nonlocal minimal surfaces, and phase transition models.

In this way, the fractional De Giorgi conjecture links the theory of nonlocal operators with classical rigidity results, extending the geometric and analytic insights of De Giorgi's original problem to a broader nonlocal framework.

\subsection{The Gross-Pitaevskii system and our contribution}

In the first paper \cite{WZ25} of this series, we have extended the De Giorgi conjecture to the following local Gross-Pitaevskii equation:
\begin{equation*}
    \left\{\begin{aligned}
        \Delta u=u(u^{2}+v^{2}-1)+v(\alpha uv-\omega),\\
        \Delta v=v(u^{2}+v^{2}-1)+u(\alpha uv-\omega),
    \end{aligned}
    \right.\quad u,v>0.
\end{equation*}
We  proved that monotonic global solutions in $\mathbb{R}^{3}$ to such a system must be one-dimensional. 

The system arises naturally in the physics of two-component Bose-Einstein condensates (BECs) exhibiting partial phase transition. It is derived from the Gross-Pitaevskii energy functional describing two-component condensates with both intra- and interspecies interactions. In contrast to the classical two-component segregated BEC models, the present system couples the two equations both linearly and nonlinearly, due to the spin coupling of the hyperfine states in addition to the intercomponent interaction. Here $u$ and $v$ represent the wave functions of the two components, $\alpha>0$ denotes the interaction parameter, and $\omega$ is the Rabi frequency corresponding to the one-body coupling between the two internal states. Such systems have attracted considerable attention in mathematical physics and PDE theory, since they exhibit rich phenomena such as phase separation, symmetry breaking, and interface dynamics \cite{AmbrosettiColorado2007,BaoCai2013, WeiWeth2010}.

In the one-dimensional case ($n=1$), the system was analyzed in \cite{AftalionSourdis2019}, where the existence and asymptotic behavior of domain wall solutions were established under different regimes of the parameter $a$. This analysis revealed the heteroclinic structure of the problem and its connection to the minimization of the Gross-Pitaevskii energy, providing a natural motivation to study monotonicity and uniqueness of solutions in higher dimensions. From a physical perspective, understanding the qualitative properties of \eqref{eq. main} constitutes a first step toward the mathematical analysis of vortex patterns in Rabi-coupled condensates, particularly the so-called multidimer bound states~\cite{AftalionMason2016,CiprianiNitta2013,KobayashiNitta2014}. These are molecular-type states formed by the binding of vortices from different components, which then interact to generate a variety of complex lattice configurations, such as honeycomb, triangular, or square patterns. The possibility of non-segregated states at infinity with positive limits underlies the formation of such rich structures.

In this paper, we extend our classification results to the following nonlocal system:
\begin{equation}\label{eq. main}
    \left\{
    \begin{aligned}
        (-\Delta)^{s}u + u(u^{2}+v^{2}-1) + v(\alpha uv - \omega) &= 0,\\
        (-\Delta)^{s}v + v(u^{2}+v^{2}-1) + u(\alpha uv - \omega) &= 0,
    \end{aligned}
    \right. \quad u,v>0,
\end{equation}
which can be seen as a fractional analogue of the local Gross-Pitaevskii system studied in \cite{WZ25}. This nonlocal model describes two-component Rabi-coupled Bose-Einstein condensates with long-range interactions, retaining both linear and nonlinear couplings present in the local system. Our aim is to classify monotone entire solutions of \eqref{eq. main} in $\mathbb{R}^{n}$, thereby extending classical and fractional De Giorgi conjecture  results to a coupled, nonlocal setting and providing insight into phase transitions and vortex patterns in multi-component condensates.

\subsection{Notations}

Before stating our main results, we first clarify a few notations and provide some background on the fractional operators involved.

The symbol $(-\Delta)^s$ in \eqref{eq. main} denotes the fractional Laplacian, a nonlocal operator that generalizes the classical Laplacian and accounts for long-range interactions. It can be expressed as
\begin{equation*}
(-\Delta)^s u(x) = C_{n, s}\, PV \int_{\mathbb{R}^n} \frac{u(x)-u(y)}{|x-y|^{n+2s}}\,dy 
= C_{n, s} \lim_{\varepsilon \to 0} \int_{\mathbb{R}^n \setminus B_\varepsilon(x)} \frac{u(x)-u(y)}{|x-y|^{n+2s}}\,dy,
\end{equation*}
where $PV$ denotes the Cauchy principal value, and $C_{n, s}$ is a dimensional normalization constant depending on $n$ and $s$. To ensure the integral is well defined, we require that 
\[
u \in C_{\rm loc}^{1,1}(\mathbb{R}^n) \cap \mathcal{L}_{2s}(\mathbb{R}^n), \quad \text{where} \quad 
\mathcal{L}_{2s}(\mathbb{R}^n) = \Big\{ u \in L_{\rm loc}^1(\mathbb{R}^n) : \int_{\mathbb{R}^n} \frac{|u(x)|}{1+|x|^{n+2s}}\, dx < \infty \Big\},
\]
endowed with the norm
\[
\|u\|_{\mathcal{L}_{2s}(\mathbb{R}^n)} := \int_{\mathbb{R}^n} \frac{|u(x)|}{1+|x|^{n+2s}}\, dx.
\]
This function space ensures both local regularity and sufficient decay at infinity, which are crucial for the well-posedness of nonlocal problems.

A convenient tool to study the fractional Laplacian is the Caffarelli-Silvestre extension~\cite{CS07}, which represents $(-\Delta)^s$ as a Dirichlet-to-Neumann map of a local degenerate elliptic problem in the upper half-space. The extension $U: \mathbb{R}^{n+1}_+ \to \mathbb{R}$ of $u$ is defined via convolution with the Poisson kernel:
\[
U(x,y) = \int_{\mathbb{R}^n} P_s(x-z,y)\, u(z)\, dz, \quad 
P_s(x,y) = c_{n,s} \frac{y^{2s}}{(|x|^2 + y^2)^{\frac{n+2s}{2}}}, \quad (x,y) \in \mathbb{R}^n \times (0,\infty).
\]
For globally bounded and sufficiently smooth functions $u$, the extended function $U$ is smooth in the extended space and satisfies
\[
\begin{cases}
\operatorname{div} \big( y^{1-2s} \nabla U(x,y) \big) = 0, & (x,y) \in \mathbb{R}^{n+1}_+, \\[1mm]
U(x,0) = u(x), & x \in \mathbb{R}^n.
\end{cases}
\]
Moreover, $(-\Delta)^s u$ can be realized as the Dirichlet-to-Neumann operator:
\[
(-\Delta)^s u(x) = -d_s \lim_{y \to 0^+} y^{1-2s} \, \partial_y U (x,y), \qquad x \in \mathbb{R}^n,
\]
where the normalization constant is 
\[
d_s = \frac{2^{2s-1} \Gamma(s)}{\Gamma(1-s)}.
\]

Regarding the nonlocal system \eqref{eq. main}, it is important to note that when the parameters $(\alpha, \omega)$ satisfy $0<\omega<\frac{\alpha}{2}$, there exist exactly three distinct positive constant solutions:
\[
(a,b), \quad (b,a), \quad \text{and} \quad \Big(\sqrt{\tfrac{1+\omega}{2+\alpha}}, \, \sqrt{\tfrac{1+\omega}{2+\alpha}}\Big),
\]
which correspond to homogeneous equilibrium states of the system. These solutions play a key role in understanding the structure of monotone solutions and the formation of interfaces in Rabi-coupled condensates (see Figure~\ref{fig. range} for a schematic representation).

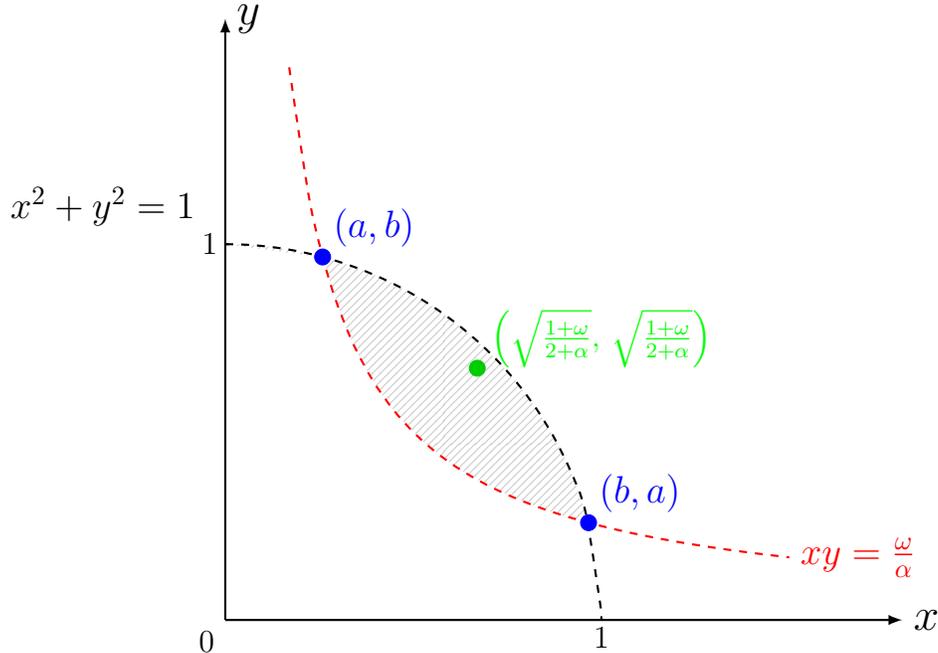
\begin{figure}[htbp]
    \centering
    \begin{tikzpicture}[scale=5,>=latex]
      \draw[->,thick] (0,0) -- (1.8,0) node[font=\Large][right] {$x$};
      \draw[->,thick] (0,0) -- (0,1.6) node[font=\Large][right] {$y$};
       \draw[thick] (1,0) -- (1,0.01) node[below] {$1$};
\draw[thick] (0, 1) -- (0.01,1) node[left] {$1$};
      \node[below left] at (0,0) {$0$};

      \pgfmathsetmacro{\aone}{sqrt(0.5 - sqrt(3)/4)}
      \pgfmathsetmacro{\bone}{1/(4*\aone)}
      \pgfmathsetmacro{\atwo}{sqrt(0.5 + sqrt(3)/4)}
      \pgfmathsetmacro{\btwo}{1/(4*\atwo)}

\draw[white, dashed] plot[domain=0:\atwo, variable=\x, smooth] ({\x},{sqrt(1-\x*\x)});
\draw[white, dashed] plot[domain=\aone:\atwo, variable=\x, smooth] ({\x},{1/(4*\x)});
      \begin{scope}
        \clip plot[domain=0:\atwo, variable=\x, smooth] ({\x},{sqrt(1-\x*\x)})
             -- plot[domain=\atwo:\aone, variable=\x, smooth] ({\x},{1/(4*\x)}) -- cycle;
        \fill[pattern=north east lines, pattern color=gray!40] (0,0) rectangle (1.6,1.6);
      \end{scope}

      \draw[dashed, thick, domain=0:1, variable=\x, smooth]
        plot ({\x},{sqrt(1-\x*\x)}) node[font=\large][left] at (-0.05, 1.1) {$x^2+y^2=1$};

      \draw[dashed, red, thick, domain=0.17:1.5, variable=\x, smooth]
        plot ({\x},{1/(4*\x)}) node[font=\large][right] {$xy=\tfrac{\omega}{\alpha}$};

      \filldraw[blue] (\aone,\bone) circle(0.6pt) node[font=\large][above right] {$(a,b)$};
      \filldraw[blue] (\atwo,\btwo) circle(0.6pt) node[font=\large][above right] {$(b,a)$};

      \filldraw[green!80!black] (0.67,0.67) circle(0.6pt);
      \node[green, font=\fontsize{12}{12}\selectfont] at (1,0.75) 
      {$\Bigl(\sqrt{\tfrac{1+\omega}{2+\alpha}},\,\sqrt{\tfrac{1+\omega}{2+\alpha}}\Bigr)$};

    \end{tikzpicture}
\caption{Steady states $(a, b), (b, a)$ and  $\Bigl(\sqrt{\tfrac{1+\omega}{2+\alpha}},\,\sqrt{\tfrac{1+\omega}{2+\alpha}}\Bigr)$.}\label{fig. range}
\end{figure}

\subsection{Main results}



Let us now state our main result, which is a Liouville-type theorem of De Giorgi type in $\mathbb{R}^{3}$.

\begin{theorem}\label{thm. main theorem}
Let $(\alpha,\omega)$ be a pair of parameters satisfying $\displaystyle 0<\omega<\frac{\alpha}{2}$. Let $(n,s)$ be a pair such that either $s\in[\frac{1}{2},1)$ with $n\geq3$, or $s\in(0,\frac{1}{2})$ with $n\geq2$. Assume that $(u,v)$ is a pair of positive and globally bounded functions in $C^{1,1}_{\rm loc}(\mathbb{R}^{n})$ satisfying \eqref{eq. main}. Suppose further that
\[
\frac{\partial u}{\partial x_{3}} > 0 > \frac{\partial v}{\partial x_{3}} \quad \text{in } \mathbb{R}^{n}.
\]
Then the pair $(u,v)$ must be one-dimensional. More precisely, there exists a unit vector $\vec{\nu}\in\mathbb{R}^{n}$ and a pair of positive functions $(U(t),V(t))$ satisfying the one-dimensional version of \eqref{eq. main} in $\mathbb{R}$, such that
\[
(u(x),v(x)) \equiv (U(x\cdot \vec{\nu}), V(x\cdot \vec{\nu})), \quad 
U'(t) > 0 > V'(t) \quad \text{for all } t \in \mathbb{R}.
\]
Moreover, let $0<a<b$ be constants such that $a^{2}+b^{2}=1$ and $\displaystyle ab = \frac{\omega}{\alpha}$. Then the asymptotic limits are
\[
\lim_{t\to -\infty} (U(t),V(t)) = (a,b), \quad
\lim_{t\to +\infty} (U(t),V(t)) = (b,a),
\]
describing a transition layer connecting the two homogeneous equilibria.
\end{theorem}

From a broader perspective, this development establishes a natural connection between the theory of nonlocal operators and the analysis of phase transitions in multi-component Rabi-coupled Bose--Einstein condensates. 
Analogous results for non-cooperative elliptic systems in the local case $s=1$ can be found in \cite{FarinaSciunziSoave2020, Fazly2021JDE,Fazly2013CVPDE, Le2023}, and results for a nonlocal system can be found in \cite{Fazly2017,FS15}.
For the competitive system \eqref{eq. main} arising in two-component Bose-Einstein condensates, the Gibbons-type conjecture was recently proved in all dimensions by Aftalion, Farina, and Nguyen \cite{AFN21}. 
Our approach is inspired by Ambrosio and Cabr\'e~\cite{AC00} for monotone solutions of the Allen-Cahn equation, which combines energy estimates, sliding techniques, and careful asymptotic analysis. 
Extending these methods to the fractional and coupled setting introduces new challenges due to the nonlocality and the interactions between components. 
We overcome these difficulties through suitable adaptations, providing a rigorous characterization of monotone solutions and the structure of the associated transition layers in Rabi-coupled Bose-Einstein condensates with long-range interactions.

\section{Preliminaries}

In this section, we collect several auxiliary results that will be repeatedly used in the sequel. 
They mainly concern integral estimates and structural properties of nonnegative functions in the space 
$\mathcal{L}_{2s}(\mathbb{R}^n)$.
Such estimates play a fundamental role in the analysis of nonlocal operators of fractional order, especially when establishing decay, growth, and comparison-type inequalities.

\subsection{Basic estimates}

We first show that  the following Ricatti-type equation has no non-trivial positive solution.

\begin{lemma}\label{lem. Ricatti}
    Assume that $u\in\mathcal{L}_{2s}(\mathbb{R}^{n})$ is a non-negative, continuous, and bounded function satisfying
    \begin{equation*}
        (-\Delta)^{s}u\leq-\lambda u^{p}
    \end{equation*}
    in the weak or viscosity sense, where $\lambda,p>0$ are two given constants. Then $u\equiv0$ in $\mathbb{R}^n$.
\end{lemma}
\begin{proof}
    Let us denote $\displaystyle M:=\sup_{x\in\mathbb{R}^{n}}u$, which is a finite and non-negative number. Our goal is to show $M=0$. Suppose on the contrary that $M>0$, we first apply the weak Harnack principle (see \cite[Theorem 1.2]{dCKP14}) to $v=M-u$, which satisfies $(-\Delta)^{s}v\geq0$ and $v\geq0$ in $\mathbb{R}^{n}$. It then follows that for every $r>0$ and $p\in\mathbb{R}^{n}$,
    \begin{equation*}
        \Big(\frac{1}{|B_{r}(p)|}\int_{B_{r}(p)}v^{t}dx\Big)^{1/t}\leq C\inf_{B_{2r}(p)}v,\quad\mbox{where }t<\frac{n}{n-2s}.
    \end{equation*}

    Let $p$ be chosen such that $u(p)\geq(1-\epsilon)M$ for a sufficiently small $\epsilon$, which means $\displaystyle\inf_{B_{2r}(p)}v\leq\epsilon M$ for any $r>0$. It then follows that
    \begin{equation*}
        C\cdot\epsilon M\geq\Big(\frac{1}{|B_{r}(p)|}\int_{B_{r}(p)}v^{t}dx\Big)^{1/t}\geq\frac{M}{2}\cdot\Big|B_{r}(p)\cap\{x:u(x)\leq\frac{M}{2}\}\Big|^{1/t}\cdot|B_{r}(p)|^{-1/t}
    \end{equation*}
    Therefore, by choosing a suitable $p$ such that the corresponding $\epsilon$ is small, we have that for all $r>0$,
    \begin{equation*}
        \Big|\Omega_{r}\Big|\geq\frac{1}{2}\cdot|B_{r}(p)|,\quad\Omega_{r}:=B_{r}(p)\cap\{x:u(x)\geq\frac{M}{2}\}.
    \end{equation*}
    Write
    \begin{equation*}
        \Lambda:=\lambda\cdot(\frac{M}{2})^{p}>0,
    \end{equation*}
    then we have $0\leq u\leq M$ everywhere in $\mathbb{R}^{n}$ and it satisfies:
    \begin{equation*}
        (-\Delta)^{s}u\leq-\Lambda\cdot\chi_{\Omega_{r}},\quad\mbox{for every }r>0.
    \end{equation*}

    By the maximal principle, we see that for every $r>0$, $u(x)\leq M+w_{r}(x)$, where $w_{r}(x)$ is the solution to the following problem:
    \begin{equation*}
        \left\{\begin{aligned}
            &(-\Delta)^{s}w_{r}(x)=-\Lambda\cdot\chi_{\Omega_{r}},&\mbox{when }&x\in B_{r}(p),\\
            &w_{r}(x)=0,&\mbox{when }&x\notin B_{r}(p),
        \end{aligned}\right.
    \end{equation*}
    Let $G_{r}(x,y)$ be the Green function in $B_{r}(p)$ of the operator $(-\Delta)^{s}$, then we have that for $x=p$:
    \begin{align*}
        G_{r}(p,y)=&\kappa_{n,s}|p-y|^{2s-n}\int_{0}^{\frac{r^2-|p-y|^2}{|p-y|^2}}\frac{t^{s-1}}{(1+t)^{n/2}}dt=r^{2s-n}g(\frac{|p-y|}{r})
    \end{align*}
    where $g(\rho)$ is a decreasing function in $\rho$ with $g(1)=0$. Then, we have
    \begin{equation*}
        w_{r}(p)\leq-\Lambda\int_{B_{r}(p)}G_{r}(p,y)\cdot\chi_{\Omega_{r}}dy\leq-\Lambda\int_{B_{r}(p)\setminus B_{2^{-\frac{1}{n}}r}(p)}G_{r}(p,y)dy=-c\cdot r^{2s}.
    \end{equation*}
    In other words, we have $u(p)\leq M-c\cdot r^{2s}$ for every $r>0$, which contradicts the assumption that $u\geq0$.
\end{proof}

In order to control the boundary behavior of the extension functions and to connect the nonlocal energy with its local counterpart, 
we need the following  trace inequality.

\begin{lemma}[Trace inequality]\label{lem. trace}
    Assume that $f(x,y)\geq0$ is a compactly supported function in $\mathbb{R}^{n+1}_{+}$. Then 
    \begin{equation*}
        \|f(x,0)\|_{H^{s}(\mathbb{R}^{n})}^{2}\leq C\int_{\mathbb{R}^{n+1}_{+}}y^{1-2s}|\nabla f|^{2}dxdy.
    \end{equation*}
\end{lemma}
\begin{proof}
    Without loss of generality, we  assume that $\|f(x,0)\|_{H^{s}(\mathbb{R}^{n})}=1$ and that $f(x,y)$ is supported in $[-1,1]^{n}\times[0,1]$. Let $g(x,y)\geq0$ denote the Caffarelli-Silvestre extension of $f(x,0)$. Then
    \begin{equation*}
        \|f(x,0)\|_{H^{s}(\mathbb{R}^{n})}^{2}\leq C\int_{\mathbb{R}^{n+1}_{+}}y^{1-2s}|\nabla g|^{2}dxdy.
    \end{equation*}
   It suffices to show that for any extension $f(x,y)$,
    \begin{equation}\label{eq. arbitrary extension has larger energy}
        \int_{\mathbb{R}^{n+1}_{+}}y^{1-2s}|\nabla f|^{2}dxdy\geq\int_{\mathbb{R}^{n+1}_{+}}y^{1-2s}|\nabla g|^{2}dxdy.
    \end{equation}
    
  Since $f(x,0)$ is supported in $[-1,1]^n$, we have 
    \begin{equation*}
        |g(x,y)|\leq\frac{C_{1}}{(|x|^{2}+y^{2})^{n/2}},\quad\mbox{as long as }\max\{|x_{1}|,\cdots,|x_{n}|,y\}\geq2.
    \end{equation*}
    We define the following function $h_{R}$ in $[-R-1,R+1]^{n}\times[0,R+1]$:
    \begin{equation*}
        h_{R}(x,y)=\frac{C_{2}}{R^{n}}\Big(\max\{|x_{1}|,\cdots,|x_{n}|,y\}-R\Big)_{+}.
    \end{equation*}
   By choosing $C_{2}\gg C_{1}$, we have  $h_{R}(x,y)\geq g(x,y)$ on  the boundary $\partial([-R-1,R+1]^{n}\times[0,R+1])$. Define the region $\Omega_{R}$ by
    \begin{equation*}
        \Omega_{R}=\{(x,y)\in[-R-1,R+1]^{n}\times[0,R+1]:\ h_{R}(x,y)\leq g(x,y)\}.
    \end{equation*}
    In $\Omega_{R}$, set
    \begin{equation*}
        f_{R}(x,y)=f(x,y)+h_{R}(x,y).
    \end{equation*}
    At least for $R\geq3$, we have $f_{R}(x,y)\equiv f(x,y)$ in $[-2,2]^{n}\times[0,2]$.

    Since $g(x,y)$ is the Caffarelli-Silvestre extension of $f(x,0)$ and agrees with $f_R$ on $\partial \Omega_R$, the minimality of the extension yields
    \begin{equation*}
        \int_{\Omega_{R}}y^{1-2s}|\nabla f_{R}|^{2}dxdy\geq\int_{\Omega_{R}}y^{1-2s}|\nabla g|^{2}dxdy.
    \end{equation*}
    On the other hand, as $\displaystyle|\nabla h_{R}|\leq\frac{C}{R^{n}}$ outside $[-R,R]^{n}\times[0,R]$, for large $R$ we have 
    \begin{equation*}
        \int_{\Omega_{R}\setminus([-2,2]^{n}\times[0,2])}y^{1-2s}|\nabla f_{R}|^{2}dxdy\leq\frac{C}{R^{2n}}\Big(R^{1-2s}R^{n}+R^{n-1}\int_{0}^{R}y^{1-2s}dy\Big)\leq\frac{C}{R^{n+2s-1}}=o(1).
    \end{equation*}
   Taking $R \to \infty$, we obtain
    \begin{align*}
        \int_{\mathbb{R}^{n+1}_{+}}y^{1-2s}|\nabla f|^{2}dxdy=&\int_{\Omega_{R}}y^{1-2s}|\nabla f_{R}|^{2}dxdy-\int_{\Omega_{R}\setminus([-2,2]^{n}\times[0,2])}y^{1-2s}|\nabla f_{R}|^{2}dxdy\\
        \geq&\int_{\Omega_{R}}y^{1-2s}|\nabla g|^{2}dxdy-o(1).
    \end{align*}
    Letting $R \to \infty$, we deduce  \eqref{eq. arbitrary extension has larger energy} and complete the proof of Lemma~\ref{lem. trace}.
\end{proof}

\subsection{Range of the solution}
In the next step, we analyze the pointwise range of {\color{red}bounded}  positive solutions to system~\eqref{eq. main}, a key ingredient in establishing the subsequent compactness and regularity results.

\begin{lemma}\label{leuaddv}
Let $u,v>0$ be two globally bounded functions solving the system \eqref{eq. main}. Then there exists some $q>0$, such that
\[
q^{-1}\leq u(x),v(x)\leq q\quad \text{for all } x \in \mathbb{R}^n.
\]
\end{lemma}
\begin{proof}
    Since $u$ and $v$ are both bounded, we have $u,v\leq q_{1}$ for some sufficiently large $q_{1}$. Next, we argue that
    \begin{equation*}
        u,v\geq\min\Big\{\frac{1}{2},\frac{\omega}{(1+\alpha)q_{1}}\Big\}=:q_{2},
    \end{equation*}
    which confirms the existence of the desired $q>0$.

    Suppose on the contrary that $\Omega:=\{x:\ u(x)\leq q_{2}\}$ is nonempty (and we will argue similarly for $v$), then for $x\in\Omega$, we have
    \begin{align}\label{eq. value of laplace if u too small}
        (-\Delta)^{s}u=&(1-u^{2})u+\omega v-(1+\alpha)u\cdot v^{2}\\
        \geq&(1-q_{2}^{2})u+\min\Big\{\omega\cdot0-(1+\alpha)u\cdot0^{2},\omega\cdot q_{1}-(1+\alpha)u\cdot q_{1}^{2}\Big\}\\
        \geq&\frac{3}{4}u+\min\Big\{0,\omega\cdot q_{1}-(1+\alpha)\cdot\frac{\omega}{(1+\alpha)q_{1}}\cdot q_{1}^{2}\Big\}=\frac{3}{4}u.
    \end{align}

    Let $\overline{x}\in\Omega$. Then, we let $R>0$ be chosen such that there exists an eigenfunction $\varphi(y)\geq0$ with $\varphi(\overline{x})=\|\varphi\|_{L^{\infty}(\mathbb{R}^{n})}=1$ solving
    \begin{equation*}
        \left\{\begin{aligned}
            &(-\Delta)^{s}\varphi=\frac{1}{2}\varphi&\mbox{in }&B_{R}(\overline{x}),\\
            &\varphi=0&\mbox{in }&B_{R}^{c}(x^{*}).
        \end{aligned}\right.
    \end{equation*}
    We increase $\lambda$ from $0$ until the graph of $\lambda\varphi$ first touches the graph of $u$. Denote this value by $\lambda^{*}$. Then, as $\lambda^{*}\varphi(\overline{x})\leq u(\overline{x})\leq q_{2}$, we have $0<\lambda^{*}\leq q_{2}$. Let $y^{*}$ be the contact point of $\lambda^{*}\varphi$ and $u$, then
    \begin{equation*}
        u(y^{*})=\lambda^{*}\varphi(y^{*})\leq\lambda^{*}\leq q_{2}.
    \end{equation*}
    In other words, we also have $y^{*}\in\Omega$. Using the estimate \eqref{eq. value of laplace if u too small}, we have
    \begin{equation*}
        (-\Delta)^{s}u(y^{*})\geq\frac{3}{4}u(y^{*}).
    \end{equation*}
    On the other hand, we have from the construction of $\varphi$ that
    \begin{equation*}
        (-\Delta)^{s}(\lambda^{*}\cdot\varphi)(y^{*})=\frac{1}{2}(\lambda^{*}\cdot\varphi)(y^{*}).
    \end{equation*}
    We then see that
    \begin{equation*}
        (-\Delta)^{s}(u-\lambda^{*}\cdot\varphi)(y^{*})\geq\frac{3}{4}u(y^{*})-\frac{1}{2}(\lambda^{*}\cdot\varphi)(y^{*})=\frac{1}{4}(\lambda^{*}\cdot\varphi)(y^{*})>0,
    \end{equation*}$u-$
    Hence, $(\sigma-\lambda^{*}\cdot\varphi)(y)$ attains its minimum at $y=y^{*}$, and 
    \begin{equation*}
        (-\Delta)^{s}(\sigma-\lambda^{*}\cdot\varphi)(y^{*})\geq\frac{1}{4}(\lambda^{*}\cdot\varphi)(y^{*})>0,
    \end{equation*}
    which contradicts the fact that $(\sigma-\lambda^{*}\cdot\varphi)$ attains its minimal value at $y^{*}$. This completes the proof of Lemma \ref{leuaddv}.
\end{proof}
Based on Lemma \ref{leuaddv}, $(-\Delta)^s \ln u$ and $(-\Delta)^s \ln v$ are well-defined. This allows us to provide a more precise characterization of the admissible range of $(u,v)$, which will be crucial for the subsequent estimates.

\begin{lemma}\label{lem. range}
 Let $(u,v)$ be a positive solution to \eqref{eq. main}. Then
 \begin{equation*}
     u^{2}+v^{2}\leq1\mbox{ and }uv\geq\frac{\omega}{\alpha}\mbox{ in }\mathbb{R}^n.
 \end{equation*}
\end{lemma}
\begin{proof}
    Similar to \cite[Proposition 1.2]{AFN21}, we define
    \begin{equation*}
        \xi=u^{2}+v^{2},\quad\eta=\ln{u}+\ln{v}.
    \end{equation*}
    It suffices to show $\xi\leq1$ and $\eta\geq\ln{\frac{\omega}{\alpha}}$. To this end, we denote
    \begin{equation*}
        M=\sup_{\mathbb{R}^{n}}\xi,\quad M_{*}=\max\{M,1\},\quad N=\inf_{\mathbb{R}^{n}}\eta,\quad N_{*}=\min\{N,\ln{\frac{\omega}{\alpha}}\}.
    \end{equation*}
    
    We first prove the following two inequalities:
    \begin{equation}\label{eq. product rule}
        (-\Delta)^{s}\xi\leq2u(-\Delta)^{s}u+2v(-\Delta)^{s}v,\quad(-\Delta)^{s}\eta\geq\frac{(-\Delta)^{s}u}{u}+\frac{(-\Delta)^{s}v}{v}.
    \end{equation}
    Indeed, for a fixed $x \in \mathbb{R}^n$, we use the inequality $a^2 - b^2 \le 2a^2 - 2ab$ with $a = u(x)$ and $b = u(y)$ to obtain
    \begin{align*}
        &\Big((-\Delta)^{s}u^{2}\Big)(x)=C_{n,s}PV\int_{\mathbb{R}^{n}}\frac{u^{2}(x)-u^{2}(y)}{|x-y|^{n+2s}}\,dy\\
        \leq&C_{n,s} PV\int_{\mathbb{R}^{n}}\frac{2u^{2}(x)-2u(x)u(y)}{|x-y|^{n+2s}}\,dy=u(x)(-\Delta)^{s}u(x),
    \end{align*}
    which gives the first inequality in \eqref{eq. product rule}.  Similarly, using $\ln{a}\geq1-\frac{1}{a}$ for $a>0$, we  set $a=\frac{u(x)}{u(y)}$, which is well-defined and positive, then
    \begin{align*}
        &\Big((-\Delta)^{s}\ln{u}\Big)(x)=C_{n,s}PV\int_{\mathbb{R}^{n}}\frac{\ln{u}(x)-\ln{u}(y)}{|x-y|^{n+2s}}dy\\
        \geq&C_{n,s}PV\int_{\mathbb{R}^{n}}\frac{1-\frac{u(x)}{u(y)}}{|x-y|^{n+2s}}dy=\frac{(-\Delta)^{s}u(x)}{u(x)}.
    \end{align*}
    This implies the second inequality in \eqref{eq. product rule}.

    It suffices to show
    \begin{equation}\label{eq. range estimate statement 1}
        M\leq\frac{1}{2}(1+\sqrt{1-8s_{*}}),\quad\mbox{where }s_{*}=\min_{t\geq e^{N_{*}}}(\alpha t^{2}-\omega t),
    \end{equation}
    and
    \begin{equation}\label{eq. range estimate statement 2}
        N\geq\ln{\frac{\omega}{\alpha+2-\frac{2}{M_{*}}}}.
    \end{equation}
    Then, following steps 3 and 4 in \cite[Proposition 1.2]{AFN21}, the desired range estimates
\[
u^2 + v^2 \le 1, \quad uv \ge \frac{\omega}{\alpha}
\]
follow.

    \textbf{Proof of \eqref{eq. range estimate statement 1}.} Employing the first inequality of \eqref{eq. product rule}, we have
    \begin{align*}        
        (-\Delta)^{s}\xi\leq&-2(u^{2}+v^{2})(u^{2}+v^{2}-1)-4uv(\alpha uv-\omega)\\
        =&-2\xi^{2}-2\xi-4\alpha e^{2\eta}+4\omega e^{\eta}\leq-2\xi^{2}-2\xi-4s_{*},
    \end{align*}
    where we have used $\eta\geq N\geq N_{*}$ in the last step. If $\xi$ is somewhere larger than $\frac{1}{2}(1+\sqrt{1-8s_{*}})$, then in the region where $\xi>\frac{1}{2}(1+\sqrt{1-8s_{*}})$, we have
    \begin{align*}
        (-\Delta)^{s}\Big(\xi-\frac{1}{2}(1+\sqrt{1-8s_{*}})\Big)_{+}\leq&(-\Delta)^{s}\Big(\xi-\frac{1}{2}(1+\sqrt{1-8s_{*}})\Big)\\
        \leq&-2\Big(\xi-\frac{1}{2}(1+\sqrt{1-8s_{*}})\Big)\Big(\xi-\frac{1}{2}(1-\sqrt{1-8s_{*}})\Big)\\
        \leq&-2\Big(\xi-\frac{1}{2}(1+\sqrt{1-8s_{*}})\Big)^{2}.
    \end{align*}
    This contradicts Lemma~\ref{lem. Ricatti}.

    \textbf{Proof of \eqref{eq. range estimate statement 2}.} By the second inequality of \eqref{eq. product rule}, we have
    \begin{equation*}
        (-\Delta)^{s}\eta\geq2-(2+\alpha)(u^{2}+v^{2})+\omega\frac{u^{2}+v^{2}}{uv}=2-(2+\alpha)\xi+\omega\xi e^{-\eta}.
    \end{equation*}
    If $\eta$ is somewhere smaller than $\ln{\frac{\omega}{\alpha+2-\frac{2}{M_{*}}}}$, then in the region, setting $\eta\leq\ln{\frac{\omega}{\alpha+2-\frac{2}{M_{*}}}}$, we let $\displaystyle\widetilde{\eta}=\Big(\ln{\frac{\omega}{\alpha+2-\frac{2}{M_{*}}}}-\eta\Big)_{+}$ and have that
    \begin{align*}
        (-\Delta)^{s}\widetilde{\eta}\leq&(2+\alpha)\xi-2-(\alpha+2-\frac{2}{M_{*}})\xi e^{\widetilde{\eta}}=\frac{2\xi}{M_{*}}-2+(\alpha+2-\frac{2}{M_{*}})\xi\cdot(1-e^{\widetilde{\eta}})\\
        \leq&-(\alpha+2-\frac{2}{M_{*}})\xi\cdot\frac{\widetilde{\eta}^{2}}{2}\leq-\alpha\cdot\inf_{\mathbb{R}^{n}}\xi\cdot\frac{\widetilde{\eta}^{2}}{2}.
    \end{align*}
    Here, we have used the fact $M_{*}\geq1$. As we have shown that $u$ and $v$ are bounded from below in Lemma~\ref{leuaddv}, we see that $\alpha\cdot\inf_{\mathbb{R}^{n}}\xi>0$. Applying Lemma~\ref{lem. Ricatti} again yields a contradiction. This means tha $\eta\geq\ln{\frac{\omega}{\alpha+2-\frac{2}{M_{*}}}}$ everywhere. 
\end{proof}

Let us define an infinitesimal potential energy as follows:
\begin{equation}\label{eq. double well infinitesimal potential energy}
    W(u,v)=\frac{1}{4}(1-u^{2}-v^{2})^{2}+\frac{1}{2\alpha}(\omega-\alpha uv)^{2}.
\end{equation}
It follows that $-P(u,v)=\frac{\partial W}{\partial u}(u,v)$ and $-Q(u,v)=\frac{\partial W}{\partial v}(u,v)$. In other words, \eqref{eq. main} can be rewritten as
\begin{equation}\label{eq. simplified B-E}
    (-\Delta)^{s}u+W_{u}(u,v)=0,\quad(-\Delta)^{s}v+W_{v}(u,v)=0.
\end{equation}
Then, it is very natural to define a coupled Ginzburg-Landau energy functional. Let $\Omega$ be an arbitrary bounded open domain in $\mathbb{R}^{n+1}_{+}$. When $s<1$, we define:
\begin{equation}\label{eq. Ginzburg-Landau energy, s<1}
    J(u,v,\Omega)=\int_{\Omega}y^{1-2s}\frac{|\nabla u|^{2}+|\nabla v|^{2}}{2}dxdy+\int_{\Omega\cap\mathbb{R}^{n}}W(u,v)dx.
\end{equation}
It is easily seen that \eqref{eq. simplified B-E} is exactly the Euler-Lagrange equation of the Ginzburg-Landau energy \eqref{eq. Ginzburg-Landau energy, s<1}.

\section{{Linearization and the Schr\"odinger Operator}}

Before analyzing the stability and fine properties of solutions to the nonlinear system \eqref{eq. main}, it is convenient to study the associated linearized problem. 
By considering the Caffarelli--Silvestre extension of $(u,v)$, the nonlocal system can be transformed into a local, weighted elliptic system in the upper half-space. 
This approach allows us to rigorously define directional derivatives, formulate a Schr\"odinger-type operator, and introduce an associated quadratic form. 
Studying these linearized objects is a crucial step toward establishing Liouville-type results, monotonicity formulas, and rigidity properties for the original nonlinear system. 
In the sequel, we develop the extension method, define the linearized operator, and establish key inequalities that will be instrumental in our analysis.

\subsection{The extension method}

Let $(U,V)$ denote the Caffarelli-Silvestre extensions of $(u,v)$, respectively. Then equation~\eqref{eq. main} can be equivalently formulated as the following local problem in the upper half-space:
\begin{equation}\label{eq. simplified extension equation}
    \left\{
    \begin{aligned}
        &\mathrm{div}(y^{1-2s}\nabla U)=0, 
        &&\lim_{y\to0}y^{1-2s}U_{y}=W_{u}(U,V),\\
        &\mathrm{div}(y^{1-2s}\nabla V)=0, 
        &&\lim_{y\to0}y^{1-2s}V_{y}=W_{v}(U,V),
    \end{aligned}
    \right.
\end{equation}
where the potential $W(u,v)$ is given by
\begin{equation*}
    W(u,v)=\frac{1}{4d_{s}}(1-u^{2}-v^{2})^{2}
    +\frac{1}{2\alpha d_{s}}(\alpha uv-\omega)^{2}.
\end{equation*}

This formulation allows one to study the nonlocal system \eqref{eq. main} through the local weighted elliptic problem \eqref{eq. simplified extension equation}. In particular, we can now analyze its linearization by differentiating in a fixed spatial direction.

\subsection{The Schr\"odinger equation}

Let $\vec{e}\in\mathbb{R}^{n}\cap\partial B_{1}$ be a unit vector. Differentiating \eqref{eq. simplified extension equation} with respect to the $\vec{e}$-direction yields the linearized (or \emph{Schr\"odinger-type}) equation satisfied by the directional derivatives.  
Before making this precise, we introduce a suitable functional class.

\begin{definition}\label{def. the class C}
    We define the class $\mathcal{C}$ to consist of all function pairs $(f,g)$ satisfying:
    \begin{itemize}
        \item $f,g\in C^{\infty}_{\mathrm{loc}}(\mathbb{R}^{n+1}_{+})\cap L^{\infty}(\overline{\mathbb{R}^{n+1}_{+}})$;
        \item Their traces on $\mathbb{R}^{n}$ satisfy $f|_{\mathbb{R}^{n}},g|_{\mathbb{R}^{n}}\in C^{\infty}_{\mathrm{loc}}(\mathbb{R}^{n})$;
        \item They solve the weighted Laplace equations 
        \[
        \mathrm{div}(y^{1-2s}\nabla f)=\mathrm{div}(y^{1-2s}\nabla g)=0
        \quad \text{in } \mathbb{R}^{n+1}_{+}.
        \]
    \end{itemize}
\end{definition}

The following lemma ensures that differentiation in the tangential direction commutes with the nonlocal Neumann operator in the extension setting.

\begin{lemma}
    Assume that $u\in C^{2,1}_{\mathrm{loc}}(\mathbb{R}^{n})$ satisfies a global $C^{1,\alpha}$ bound, i.e., $\|u\|_{C^{1,\alpha}(\mathbb{R}^{n})}<\infty$. 
    Then for any $\vec{e}\in\mathbb{R}^{n}\cap\partial B_{1}$, the $\vec{e}$-directional derivative commutes with the operator $\displaystyle\lim_{y\to0}y^{1-2s}\partial_{y}$ applied to the Caffarelli-Silvestre extension $U$ of $u$.
\end{lemma}

We are now in a position to derive the linearized equation satisfied by the directional derivatives of $(U,V)$.

\begin{lemma}\label{lem. derivative satisfies the Schrodinger equation}
    Let $(u,v)$ be a solution to \eqref{eq. simplified extension equation}, and let $\vec{e}\in\mathbb{R}^{n}\cap\partial B_{1}$ be a unit vector. 
    Then $(\partial_{\vec{e}}U,\partial_{\vec{e}}V)$ satisfies the following Schr\"odinger-type system:
    \begin{equation}\label{eq. Schrodinger of derivative}
        \left\{
        \begin{aligned}
            &\mathrm{div}(y^{1-2s}\nabla\partial_{\vec{e}}U)=0,
            &&\lim_{y\to0}y^{1-2s}\partial_{\vec{e}}U_{y}
            =W_{uu}(U,V)\,\partial_{\vec{e}}U+W_{uv}(U,V)\,\partial_{\vec{e}}V,\\
            &\mathrm{div}(y^{1-2s}\nabla\partial_{\vec{e}}V)=0,
            &&\lim_{y\to0}y^{1-2s}\partial_{\vec{e}}V_{y}
            =W_{vu}(U,V)\,\partial_{\vec{e}}U+W_{vv}(U,V)\,\partial_{\vec{e}}V.
        \end{aligned}
        \right.
    \end{equation}
\end{lemma}

\begin{proof}
Differentiating \eqref{eq. simplified extension equation} in the direction of $\vec{e}$ shows that $\partial_{\vec{e}}U$ still satisfies the weighted Laplace equation 
\[
\mathrm{div}(y^{1-2s}\nabla \partial_{\vec{e}}U) = 0 \quad \text{in } \mathbb{R}^{n+1}_+.
\] 
Moreover,  at the bottom space $\mathbb{R}^{n}$, we have
\[
(-\Delta)^s \partial_{\vec{e}} u = \partial_{\vec{e}} (-\Delta)^s u = -\partial_{\vec{e}} W_u(u,v).
\] 
Applying the chain rule then yields the first line of \eqref{eq. Schrodinger of derivative}. The second line  of \eqref{eq. Schrodinger of derivative} follows analogously.
\end{proof}

From Lemma~\ref{lem. derivative satisfies the Schrodinger equation}, it is natural to introduce the associated Schr\"odinger operator.

\begin{definition}
    Let $(U,V)$ be a solution to \eqref{eq. simplified extension equation}. We define the Schr\"odinger operator corresponding to $(U,V)$ as $L$ in the following way:
    \begin{align}\label{eq. Schrodinger}
        L\begin{bmatrix}\xi\\\eta\end{bmatrix}:=&\Big\{-\lim_{y\to0}y^{1-2s}\partial_{y}+D^{2}W\Big\}\begin{bmatrix}\xi\\\eta\end{bmatrix}\\
        =&\begin{bmatrix}-\mathop{\lim}\limits_{y\to0}y^{1-2s}\xi_{y}+W_{uu}(U,V)\xi+W_{uv}(U,V)\eta\\-\mathop{\lim}\limits_{y\to0}y^{1-2s}\eta_{y}+W_{vu}(U,V)\xi+W_{vv}(U,V)\eta\end{bmatrix}.
    \end{align}
    Here, $(\xi,\eta)$ is a pair of functions in the class $\mathcal{C}$ (see Definition~\ref{def. the class C}).
\end{definition}

Next, we define a quadratic operator associated with ratios of functions:

\begin{definition}
Let $(\varphi,\psi)$ be a pair of functions in the class $\mathcal{C}$ satisfying $\varphi>0>\psi$ everywhere in $\overline{\mathbb{R}^{n+1}_+}$, and let $(\sigma,\tau)\in \mathcal{C}$ be defined by
\[
(\sigma,\tau) := \Big(\frac{\xi}{\varphi}, \frac{\eta}{\psi}\Big) \quad \text{in } \mathbb{R}^{n+1}_+.
\]
Then, the quadratic form $Q_{(\varphi,\psi)}$ is given by
\begin{equation}\label{eq. ratio quadratic operator}
Q_{(\varphi,\psi)}\begin{bmatrix}\sigma\\\tau\end{bmatrix} 
:= \sigma^2 \varphi \lim_{y\to0} (y^{1-2s} \varphi_y) - \xi \lim_{y\to0} (y^{1-2s} \xi_y) 
+ \tau^2 \psi \lim_{y\to0} (y^{1-2s} \psi_y) - \eta \lim_{y\to0} (y^{1-2s} \eta_y).
\end{equation}
\end{definition}

In our applications, the quadratic form \eqref{eq. ratio quadratic operator} is always negative semi-definite.

\begin{lemma}\label{lem. quadratic form Q is negative semi-definite}
Assume that $(\varphi,\psi)$ is a smooth pair on $\overline{\mathbb{R}^{n+1}_+}$ with $\varphi>0>\psi$ everywhere, satisfying
\[
\mathrm{div}(y^{1-2s} \nabla \varphi) = \mathrm{div}(y^{1-2s} \nabla \psi) = 0,
\]
and
\[
L \begin{bmatrix}\varphi\\\psi\end{bmatrix} \in \overline{\mathbb{R}_+} \times \overline{\mathbb{R}_-} \quad \text{(the fourth quadrant).}
\]
Let $(\xi,\eta) \in \mathcal{C}$ lie in the kernel of \eqref{eq. Schrodinger}, with $(\sigma,\tau) = (\xi/\varphi, \eta/\psi)$. Let $Q_{(\varphi,\psi)}$ be defined as in \eqref{eq. ratio quadratic operator}. Then, there exists a strictly positive function $\mathcal{P}(x)$ such that
\begin{equation}\label{eq. Q is negative semi-definite}
Q_{(\varphi,\psi)}\begin{bmatrix}\sigma\\\tau\end{bmatrix} \le - (\sigma - \tau)^2 \mathcal{P}(x) \quad \text{everywhere in } \mathbb{R}^n.
\end{equation}
\end{lemma}

\begin{proof}
Using the positivity of $(\varphi,\psi)$, we have
\[
\varphi \lim_{y\to0} (y^{1-2s} \varphi_y) \le W_{uu} \varphi^2 + W_{uv} \varphi \psi, \quad
\psi \lim_{y\to0} (y^{1-2s} \psi_y) \ge W_{vu} \varphi \psi + W_{vv} \psi^2.
\]
Therefore,
\begin{align*}
Q_{(\varphi,\psi)}\begin{bmatrix}\sigma\\\tau\end{bmatrix} 
&\le \sigma^2 (W_{uu}\varphi^2 + W_{uv}\varphi\psi) - (W_{uu}\xi^2 + W_{uv}\xi\eta) \\
&\quad + \tau^2 (W_{vu}\varphi\psi + W_{vv}\psi^2) - (W_{vu}\xi\eta + W_{vv}\eta^2) \\
&= W_{uv} \varphi \psi (\sigma - \tau)^2 = - (\sigma - \tau)^2 \mathcal{P}(x),
\end{align*}
with $\mathcal{P}(x) := - W_{uv} \varphi \psi$.  

To see that $\mathcal{P}(x)$ is strictly positive, it suffices to check $W_{uv} > 0$. Direct computation gives
\begin{equation}\label{eq. W_uv>0}
W_{uv}(u,v) = (2 + 2 \alpha) uv - \omega \ge \Big(\frac{2}{\alpha} + 1\Big) \omega > 0,
\end{equation}
where we have used Lemma~\ref{lem. range} which asserts $uv \ge \omega / \alpha$.
\end{proof}

\subsection{The Caccioppoli inequality}

Building on the preceding constructions, we are now in a position to establish a Liouville-type result for solutions of the Schr\"odinger system. 
Such results are crucial for deriving rigidity properties, as they allow us to conclude that certain normalized ratios of solutions must be constant under appropriate energy growth conditions.

\begin{lemma}[Liouville-type result with cutoff function]\label{lem. Liouville-type result with cutoff function}
Assume that $\xi, \eta, \sigma, \tau, \varphi, \psi \in C^\infty(\mathbb{R}^{n+1}_{+})$, and they satisfy all assumptions in Lemma~\ref{lem. quadratic form Q is negative semi-definite}. Let
\begin{equation*}
    B_{r}^{+} := \{ (x,y) \in \mathbb{R}^n \times\mathbb{R}_{+}:\sqrt{|x|^{2}+y^{2}}< r \}.
\end{equation*}
Assume that the following growth condition holds:
\begin{equation}\label{eq. quadratic energy growth condition}
    \int_{B_{r}^{+}}y^{1-2s}(\xi^2+\eta^{2})\, dx dy \le C r^{2} \ln r,\quad\mbox{for all }r\geq2.
\end{equation}
Then, we have
\begin{equation*}
    \sigma(x,y)\equiv\tau(x,y)\equiv\mbox{constant},\quad\mbox{for }(x,y)\in\mathbb{R}^{n}\times\mathbb{R}_{+}.
\end{equation*}
\end{lemma}
\begin{proof}
Since $\xi = \varphi \sigma$, and 
\[
\operatorname{div} (y^{1-2s} \nabla \varphi) = \operatorname{div} (y^{1-2s} \nabla \xi) = 0 \quad \mbox{in } \mathbb{R}^n \times \mathbb{R}_+,
\]
we compute
\begin{equation}\label{eq:dive1}
\begin{aligned}
0 &= \varphi \operatorname{div}(y^{1-2s} \nabla (\sigma \varphi)) \\
  &= \varphi \Big( \sigma \operatorname{div}(y^{1-2s} \nabla \varphi) + 2 y^{1-2s} \nabla \sigma \cdot \nabla \varphi + \varphi \operatorname{div}(y^{1-2s} \nabla \sigma) \Big) \\
  &= y^{1-2s} \nabla \sigma \cdot \nabla \varphi^2 + \varphi^2 \operatorname{div}(y^{1-2s} \nabla \sigma) \\
  &= \operatorname{div}(\varphi^2 y^{1-2s} \nabla \sigma),
\end{aligned}
\end{equation}
in $\mathbb{R}^n \times \mathbb{R}_+$. Similarly, since $\eta = \psi \tau$ and 
$$\operatorname{div}(y^{1-2s} \nabla \psi) = \operatorname{div}(y^{1-2s} \nabla \eta) = 0\quad \mbox{in } \mathbb{R}^n \times \mathbb{R}_+,$$ 
we obtain
\begin{equation}\label{eq:dive2}
\operatorname{div}(\psi^2 y^{1-2s} \nabla \tau) = 0 \quad \mbox{in } \mathbb{R}^n \times \mathbb{R}_+.
\end{equation}

For $R > 2$, let $\zeta$ be a Lipschitz cut-off function
\begin{equation*}
    \zeta(x,y)=\left\{\begin{aligned}
        &1,&\mbox{if }&\sqrt{|x|^{2}+y^{2}}\le R,\\
        &2-\dfrac{\ln(\sqrt{|x|^{2}+y^{2}})}{\ln R},&\mbox{if }&R <\sqrt{|x|^{2}+y^{2}}< R^2,\\
        &0,&\mbox{if }&\sqrt{|x|^{2}+y^{2}}\ge R^2,
    \end{aligned}\right.
\end{equation*}
so that $\operatorname{supp}(\nabla\zeta)$ lies in the region where $R <\sqrt{|x|^{2}+y^{2}}< R^2$ and
\begin{equation}\label{eq:grad-bound}
|\nabla\zeta(x,y)|=\frac{\chi_{B_{R^{2}}^{+}}-\chi_{B_{R}^{+}}}{\sqrt{|x|^{2}+y^{2}}\cdot\ln{R}}.
\end{equation}

Multiplying  equations \eqref{eq:dive1} and \eqref{eq:dive2} by $ \zeta^2 \sigma$  and $ \zeta^2 \tau$, respectively, summing  the resulting identities,   and  integrating over $\mathbb{R}^n \times(0, +\infty)$ yields 
\begin{equation}\label{eq3.11}
\begin{aligned}
0=&\int\zeta^2 \sigma \operatorname{div}(\varphi^2 y^{1-2s} \nabla \sigma) \,dxdy+\int\zeta^2 \tau \operatorname{div}(\psi^2 y^{1-2s} \nabla \tau) \,dxdy\\
=&-\int y^{1-2s} \varphi^2  \nabla \sigma \cdot \nabla (\zeta^2 \sigma)\,dxdy-\int y^{1-2s} \psi^2  \nabla \tau \cdot \nabla (\zeta^2 \tau) \,dxdy\\
&-\int_{\{y=0\}} \zeta^2(x,0)\Big\{ \lim_{y\to 0}(\sigma \varphi^2 y^{1-2s} \sigma_y) + \lim_{y\to 0}(\tau\psi^2 y^{1-2s} \tau_y)\Big\} \,dx\\
\leq&-\int y^{1-2s} \varphi^2  \nabla \sigma \cdot \nabla (\zeta^2 \sigma)
\,dxdy-\int y^{1-2s} \psi^2  \nabla \tau \cdot \nabla (\zeta^2 \tau) \,dxdy\\
&-\int_{\{y=0\}} \zeta^2(x, 0) \Big(\sigma(x, 0)-\tau(x, 0)\Big)^2 \mathcal{P}(x) \,dx,
\end{aligned}
\end{equation}
where we have used \eqref{eq. Q is negative semi-definite} in the last step. Noting that $\zeta^{2}\geq\chi_{B_{R}^{+}}$, and that
\begin{equation*}
    \nabla\sigma\cdot\nabla(\zeta^{2}\sigma)\geq|\zeta\nabla\sigma|^2-2|\zeta\nabla\sigma|\cdot|\sigma\nabla\zeta|,\quad\nabla\tau\cdot \nabla(\zeta^{2}\tau)\geq|\zeta\nabla\tau|^2-2|\zeta\nabla\tau|\cdot|\tau\nabla\zeta|,
\end{equation*}
we then obtain from \eqref{eq3.11} and the fact that $\operatorname{supp}(\nabla\zeta)\subseteq(B_{R}^{+})^{c}$ that
\begin{equation}\label{eq3.11'}
\begin{aligned}
    &\int y^{1-2s}\zeta^{2}(\varphi^{2}|\nabla\sigma|^{2}+\psi^{2}|\nabla\tau|^{2})\,dxdy+\int_{\{y=0\}\cap B_{R}}|\sigma-\tau|^{2}\mathcal{P}(x)\,dx\\
    \leq&2\int y^{1-2s}(|\zeta\nabla\sigma|\cdot|\sigma\nabla\zeta|+|\zeta\nabla\tau|\cdot|\tau\nabla\zeta|)\,dxdy\\
    \leq&2\Big\{\int_{(B_{R}^{+})^{c}}y^{1-2s}\zeta^{2}(\varphi^{2}|\nabla\sigma|^{2}+\psi^{2}|\nabla\tau|^{2})\,dxdy\Big\}^{\frac{1}{2}}\Big\{\int y^{1-2s}|\nabla\zeta|^{2}(\xi^{2}+\eta^{2})\,dxdy\Big\}^{\frac{1}{2}}.
\end{aligned}
\end{equation}
In order to estimate $\displaystyle\int y^{1-2s}|\nabla\zeta|^{2}(\xi^{2}+\eta^{2})\,dxdy$, we introduce the measure
$$d \mu:=y^{1-2s}(\xi^2+\eta^2) dxdy.$$
With \eqref{eq. quadratic energy growth condition}, we see $\mu(B_{r}^{+})\leq C r^{2}\ln{r}$ for all $r\geq2$.
By \eqref{eq:grad-bound}, we have
\begin{equation*}
    \int y^{1-2s}|\nabla\zeta|^{2}(\xi^{2}+\eta^{2})\,dxdy=\int_{B_{R^{2}}^{+}\setminus B_{R}^{+}}\frac{d\mu}{(|x|^{2}+y^{2})\cdot(\ln{R})^{2}}.
\end{equation*}
For every $t\geq0$, we define
\begin{equation*}
    A(t):=\mu\Big((B_{R^{2}}^{+}\setminus B_{R}^{+})\cap\{\frac{1}{|x|^{2}+y^{2}}\geq t\}\Big)=\mu\Big((B_{R^{2}}^{+}\setminus B_{R}^{+})\cap B_{1/\sqrt{t}}^{+}\Big),
\end{equation*}
it follows from $\mu(B_{r}^{+})\leq C r^{2}\ln{r}$ that
\begin{equation*}
    A(t)\leq\left\{\begin{aligned}
        &C R^{4}\ln{R},&\mbox{if }&t\leq R^{-4},\\
        &\frac{C}{t}\ln{\frac{1}{t}},&\mbox{if }&R^{-4}\leq t\leq R^{-2},\\
        &0,&\mbox{if }&t>R^{-2}.
    \end{aligned}\right.
\end{equation*}
Consequently,
\begin{equation}\label{eq. estimate lebesgue estimate}
    \int y^{1-2s}|\nabla\zeta|^{2}(\xi^{2}+\eta^{2})\,dxdy=\int_{0}^{\infty}\frac{A(t)}{(\ln{R})^{2}}dt\leq C.
\end{equation}
Then, by combining \eqref{eq3.11'} with \eqref{eq. estimate lebesgue estimate}, we see that
\begin{equation*}
    \int_{B_{R}^{+}}y^{1-2s}(\varphi^{2}|\nabla\sigma|^{2}+\psi^{2}|\nabla\tau|^{2})\,dxdy\leq\int y^{1-2s}\zeta^{2}(\varphi^{2}|\nabla\sigma|^{2}+\psi^{2}|\nabla\tau|^{2})\,dxdy\leq C,
\end{equation*}
and we send $R\to\infty$ and have that
\begin{equation*}
    \int_{\mathbb{R}^{n}\times\mathbb{R}_{+}}y^{1-2s}(\varphi^{2}|\nabla\sigma|^{2}+\psi^{2}|\nabla\tau|^{2})\,dxdy\leq C.
\end{equation*}

Back to \eqref{eq3.11'}, using the integrability obtained above, we have
\begin{equation*}
    \lim_{R\to0}\int_{(B_{R}^{+})^{c}}y^{1-2s}(\varphi^{2}|\nabla\sigma|^{2}+\psi^{2}|\nabla\tau|^{2})\,dxdy=0.
\end{equation*}
so we get
\begin{equation*}
    \int_{B_{R}^{+}}y^{1-2s}\zeta^{2}(\varphi^{2}|\nabla\sigma|^{2}+\psi^{2}|\nabla\tau|^{2})\,dxdy+\int_{\{y=0\}\cap B_{R}}|\sigma-\tau|^{2}\mathcal{P}(x)\,dx\to0,\quad\mbox{as }R\to0.
\end{equation*}
Therefore, $\sigma$ and $\tau$ are constants. Besides, since $\mathcal{P}(x)$ is strictly positive, we also have
\begin{equation*}
\sigma(x, 0) = \tau(x, 0), \quad x \in \mathbb{R}^n.
\end{equation*}
In conclusion, $\sigma\equiv\tau\equiv\mbox{constant}$ in $\mathbb{R}^{n}\times\mathbb{R}_{+}$. This completes the proof of Lemma \ref{lem. Liouville-type result with cutoff function}.
\end{proof}
As an immediate consequence of the Liouville-type result in Lemma~\ref{lem. Liouville-type result with cutoff function}, we can now derive a more concrete classification in the setting of one dimensional lower.
\begin{corollary}\label{cor. n=2 is more obvious}
    Assume that $n\leq2$ with $\frac{1}{2}\leq s<1$, or $n\leq1$ with $0<s<\frac{1}{2}$. Assume that $(U,V)$ is a solution to \eqref{eq. simplified extension equation} in $\mathbb{R}^{n}\times\mathbb{R}_{+}$, such that there exists a pair $(\varphi,\psi)$, with $\varphi>0>\psi$ everywhere and $L\begin{bmatrix}\varphi\\\psi\end{bmatrix}$ always belonging to $\overline{\mathbb{R}_{+}}\times\overline{\mathbb{R}_{-}}$. Then
    \begin{itemize}
        \item Either: $(u,v)$ is constant.
        \item Or: There exists a unit vector $\vec{\nu}$ such that $\Big(u(x),v(x)\Big)=\Big(u(x\cdot\vec{\nu}),v(x\cdot\vec{\nu})\Big)$, with $\partial_{\vec{\nu}}u>0>\partial_{\vec{\nu}}v$ along that direction.
    \end{itemize}
\end{corollary}
\begin{proof}
    Let us denote for $1\leq i\leq n$ that:
    \begin{equation*}
        (\xi_{i},\eta_{i}):=(\frac{\partial u}{\partial x_{i}},\frac{\partial v}{\partial x_{i}}),\quad(\sigma_{i},\tau_{i}):=(\frac{\xi_{i}}{\varphi},\frac{\eta_{i}}{\psi}).
    \end{equation*}
    Accordingly, we also define $Q_{(\varphi,\psi)}\begin{bmatrix}\sigma_{i}\\\tau_{i}\end{bmatrix}$ like in \eqref{eq. ratio quadratic operator}. Then we realize that \eqref{eq. quadratic energy growth condition} is true. In fact, we recall that as $(u,v)$ is bounded, the extension $(U,V)$ is also bounded in $\mathbb{R}^{n+1}_{+}$. By the Schauder estimate for the degenerate or singular equation \eqref{eq. simplified extension equation}, we have
    \begin{equation}\label{eq. gradient pointwise estimate 1 in cor}
        \|\nabla U\|_{L^{\infty}(\mathbb{R}^{n+1}_{+})}+\|\nabla U\|_{L^{\infty}(\mathbb{R}^{n+1}_{+})}\leq C.
    \end{equation}
    Besides, we apply the interior Schauder estimate to \eqref{eq. simplified extension equation} in each ball $B_{y/2}(x,y)$. Using the uniform ellipticity of \eqref{eq. simplified extension equation} inside $B_{y/2}(x,y)$ and the boundedness of $(U,V)$, we have
    \begin{equation}\label{eq. gradient pointwise estimate 2 in cor}
        |\nabla U(x,y)|+|\nabla V(x,y)|\leq\frac{C}{y}.
    \end{equation}
    By integrating the estimates \eqref{eq. gradient pointwise estimate 1 in cor} and \eqref{eq. gradient pointwise estimate 2 in cor}, we have verified \eqref{eq. quadratic energy growth condition}.

    By Lemma~\ref{lem. Liouville-type result with cutoff function}, there exist two constants $c_{1}$ and $c_{2}$ such that
    \begin{equation*}
        \sigma_{i}\equiv\tau_{i}\equiv c_{i},\quad i=1,2.
    \end{equation*}
    By the construction of $(\sigma_{i},\tau_{i})$, we have
    \begin{equation}\label{eq. gradient of u and v in 2 dim}
        \nabla u(x)=\varphi(x)(c_{1}\vec{e}_{1}+c_{2}\vec{e}_{2}),\quad\nabla v(x)=\psi(x)(c_{1}\vec{e}_{1}+c_{2}\vec{e}_{2}).
    \end{equation}
    \begin{itemize}
        \item Case 1: If $c_{1}=c_{2}=0$, then $u(x)$ and $v(x)$ are both constants.
        \item Case 2: Otherwise, we let
        \begin{equation*}
            \vec{\nu}=\frac{c_{1}\vec{e}_{1}+c_{2}\vec{e}_{2}}{\sqrt{c_{1}^{2}+c_{2}^{2}}},\quad\vec{\nu}^{\perp}=\frac{-c_{2}\vec{e}_{1}+c_{1}\vec{e}_{2}}{\sqrt{c_{1}^{2}+c_{2}^{2}}}.
        \end{equation*}
        Then we see from \eqref{eq. gradient of u and v in 2 dim} that
        \begin{equation*}
            \vec{\nu}^{\perp}\cdot\nabla u(x)=\vec{\nu}^{\perp}\cdot\nabla v(x)=0.
        \end{equation*}
        Therefore, we have $\Big(u(x),v(x)\Big)=\Big(u(x\cdot\vec{\nu}),v(x\cdot\vec{\nu})\Big)$. Moreover, from \eqref{eq. gradient of u and v in 2 dim}, $u$ is increasing and $v$ is decreasing along $x\cdot\vec{\nu}$.
    \end{itemize}
    This completes the proof of Corollary~\ref{cor. n=2 is more obvious}.
\end{proof}

\section{Evolution of the Ginzburg-Landau energy}

We next study the growth of the Ginzburg-Landau energy in large domains. The following lemma provides an upper bound for the energy in half-balls $B_R^+$, depending on the fractional exponent $s$.

\begin{lemma}\label{lem. energy growth}
    For $R$ sufficiently large, we have
    \begin{equation*}
        J(u,v,B_{R}^{+})\leq
        \begin{cases}
            C R^{n-2s}, & \text{if } 0<s<\frac{1}{2},\\
            C R^{n-1}\ln{R}, & \text{if } s=\frac{1}{2},\\
            C R^{n-1}, & \text{if } \frac{1}{2}<s<1.
        \end{cases}
    \end{equation*}
\end{lemma}

The strategy to prove Lemma~\ref{lem. energy growth} is to push the domain $B_R$ along the $\vec{e}_n$ direction towards infinity. More precisely, we aim to establish the following two statements:

\begin{itemize}
    \item[(1)] The energy variation during the pushing procedure remains controlled and does not exceed $O(R^{n-2s})$, $O(R^{n-1}\ln{R})$, or $O(R^{n-1})$.
    \item[(2)] The limiting profile has energy similarly bounded by $O(R^{n-2s})$, $O(R^{n-1}\ln{R})$, or $O(R^{n-1})$.
\end{itemize}

In this subsection, we focus on proving the first statement, which quantifies the energy change along the pushing procedure (see Lemma~\ref{lem. evolution of E(t)} below). The second statement will be addressed in the subsequent subsection.

\begin{lemma}\label{lem. evolution of E(t)}
    Let $R$ be sufficiently large. For every $t\in\mathbb{R}$, define the cube
    \begin{equation*}
        Q_{t}=[-R,R]^{n-1}\times[t-R,t+R]\times[0,R],
    \end{equation*}
    and let $E(t)=J(u,v,Q_{t})$. If 
    $$
    \frac{\partial u}{\partial x_n}>0>\frac{\partial v}{\partial x_n} \quad \mbox{in}\quad Q_t,
    $$
    then the following estimate holds:
    \begin{equation}\label{eq4.1}
        \int_{-\infty}^{\infty}\Big|\frac{dE}{dt}\Big|\,dt\leq
        \begin{cases}
            C R^{n-2s}, & \text{if } 0<s<\frac{1}{2},\\
            C R^{n-1}\ln{R}, & \text{if } s=\frac{1}{2},\\
            C R^{n-1}, & \text{if } \frac{1}{2}<s<1.
        \end{cases}
    \end{equation}
\end{lemma}

\begin{proof}
Observe that translating the cube  $Q_t$ 
 is equivalent to shifting the solution along the $\vec{e}_n$-direction.
 Define

    \[
        (u^{(t)}(x',x_n,y), v^{(t)}(x',x_n,y)) := (u(x', x_n+t,y), v(x', x_n+t,y)).
    \]
    Then $J(u,v,Q_t) = J(u^{(t)},v^{(t)},Q_0)$ and
    \[
        \frac{d}{dt}(u^{(t)},v^{(t)}) = \left( \frac{\partial u}{\partial x_n}(x', x_n+t,y), \frac{\partial v}{\partial x_n}(x', x_n+t,y) \right).
    \]

Using integration by parts and \eqref{eq. simplified extension equation}, we obtain
\begin{eqnarray*}
\begin{aligned}
     \frac{d E}{dt}=&\frac{d}{dt}\int_{Q_t}y^{1-2s}\frac{|\nabla u|^{2}+|\nabla v|^{2}}{2}dxdy+\int_{Q_t\cap\mathbb{R}^{n}}W(u,v)dx\\
     =& \int_{Q_t}y^{1-2s} \Big\{ \nabla u \cdot \nabla \frac{\partial u}{\partial x_n} + \nabla v \cdot \nabla \frac{\partial v}{\partial x_n} \Big\}dxdy
     +\int_{Q_t\cap\mathbb{R}^{n}} \Big\{ W_u \frac{\partial u}{\partial x_n} + W_v \frac{\partial v}{\partial x_n} \Big\} dx\\
     =& -\int_{Q_t} \Big\{\operatorname{div} (y^{1-2s} \nabla u) \frac{\partial u}{\partial x_n}  +\operatorname{div} (y^{1-2s} \nabla v) \frac{\partial v}{\partial x_n} \Big\} dxdy\\
     &+\int_{\partial Q_t}\Big\{y^{1-2s} \frac{\partial u}{\partial x_n}\cdot \frac{\partial u}{\partial \vec{\nu}}+ y^{1-2s} \frac{\partial v}{\partial x_n}\cdot \frac{\partial v}{\partial \vec{\nu}}\Big\} d\mathcal{H}^{n}_{\partial Q_t}
     +\int_{Q_t\cap\mathbb{R}^{n}} \Big\{ W_u \frac{\partial u}{\partial x_n} + W_v \frac{\partial v}{\partial x_n} \Big\} dx\\
     =&- \int_{Q_t\cap \mathbb{R}^n}\lim_{y\to 0^+}\Big\{y^{1-2s}u_y\frac{\partial u}{\partial x_n}+y^{1-2s}v_y\frac{\partial u}{\partial x_n}
     \Big\} dx
     +\int_{Q_t\cap\mathbb{R}^{n}} \Big\{ W_u \frac{\partial u}{\partial x_n}  + W_v \frac{\partial v}{\partial x_n} \Big\} dx\\
     &+\int_{\partial Q_t\backslash (Q_t\cap \mathbb{R}^n)}\Big\{y^{1-2s} \frac{\partial u}{\partial x_n}\cdot \frac{\partial u}{\partial \vec{\nu}}+ y^{1-2s} \frac{\partial v}{\partial x_n}\cdot \frac{\partial v}{\partial \vec{\nu}}\Big\} d\mathcal{H}^{n}_{\partial Q_t}\\
     =&\int_{\partial Q_t\backslash (Q_t\cap \mathbb{R}^n)}\Big\{y^{1-2s} \frac{\partial u}{\partial x_n}\cdot \frac{\partial u}{\partial \vec{\nu}}+ y^{1-2s} \frac{\partial v}{\partial x_n}\cdot \frac{\partial v}{\partial \vec{\nu}}\Big\} d\mathcal{H}^{n}_{\partial Q_t},
     \end{aligned}
\end{eqnarray*}
where $\vec{\nu}$ denotes the outward unit normal on $\partial Q_t$.

 To estimate the integral on the right-hand side, we consider an arbitrary point $(x,y) = (x', x_n, y) \in \mathbb{R}^n\times(0, +\infty)$ and  decompose  the lateral boundary  $\partial Q_t\backslash (Q_t\cap \mathbb{R}^n)$  into the following  three parts:
 $$
S_1=\{(x, y) \in R^n \times(0, +\infty) \mid x\in [-R, R]^{n-1} \times [t-R, t+R], \,y=R \},
 $$
$$
S_2=\{(x, y) \in R^n \times(0, +\infty) \mid x'\in [-R, R]^{n-1}, \, x_n=t \pm R, \,y\in [0,R] \},
$$
and 
$$
S_3=\{(x, y) \in R^n \times(0, +\infty) \mid x'\in \partial [-R, R]^{n-1},\, x_n\in [t-R, t+R], \, x_n=t \pm R, \,y\in [0,R] \}.
$$
Then 
$$
\begin{aligned}
\left|\frac{d E}{dt}\right| \leq &\int_{\partial Q_t\backslash (Q_t\cap \mathbb{R}^n)}\left| y^{1-2s} \frac{\partial u}{\partial x_n}\cdot \frac{\partial u}{\partial \vec{\nu}}+ y^{1-2s} \frac{\partial v}{\partial x_n}\cdot \frac{\partial v}{\partial \vec{\nu}}\right| d\mathcal{H}^{n}_{\partial Q_t}\\
=& \int_{S_1}\left| y^{1-2s}\frac{\partial u}{\partial x_n}\cdot \frac{\partial u}{\partial \vec{\nu}}+ y^{1-2s}\frac{\partial v}{\partial x_n}\cdot \frac{\partial v}{\partial \vec{\nu}}\right| d\mathcal{H}^{n}_{\partial Q_t}\\
&+\int_{S_2}\left| y^{1-2s} \frac{\partial u}{\partial x_n}\cdot \frac{\partial u}{\partial \vec{\nu}}+ y^{1-2s} \frac{\partial v}{\partial x_n}\cdot \frac{\partial v}{\partial \vec{\nu}}\right| d\mathcal{H}^{n}_{\partial Q_t}\\
&+\int_{S_3}\left| y^{1-2s} \frac{\partial u}{\partial x_n}\cdot \frac{\partial u}{\partial \vec{\nu}}+ y^{1-2s} \frac{\partial v}{\partial x_n}\cdot \frac{\partial v}{\partial \vec{\nu}}\right| d\mathcal{H}^{n}_{\partial Q_t}\\
:=&I_1+I_2+I_3.
\end{aligned}
$$

\medskip
\noindent
\textbf{Estimate of $I_1$.} 
By Lemma \ref{lem. range}, we have that $u$ and $v$ are bounded.  Consequently, the Schauder estimates imply that
$$
\left|\frac{\partial u}{\partial y}(x, R)\right|\sim \frac{1}{R} \,\, \mbox{and}\,\, \left|\frac{\partial v}{\partial y}(x, R)\right|\sim \frac{1}{R}.
$$
Then
\begin{eqnarray*}
\begin{aligned}
I_1(t)=& R^{1-2s} \int_{t-R}^{t+R}\int_{[-R, R]^{n-1}}\left| \frac{\partial u}{\partial x_n}(x, R)\cdot \frac{\partial u}{\partial y}(x, R) + \frac{\partial v}{\partial x_n}(x, R) \cdot\frac{\partial v}{\partial y}(x,R)\right| \,dx'dx_n\\
\leq & C R^{-2s}\int_{t-R}^{t+R}\int_{[-R, R]^{n-1}} \Big\{ \frac{\partial u}{\partial x_n}(x, R)- \frac{\partial v}{\partial x_n}(x, R)\Big\} \,dx' dx_n.\\
\end{aligned}
\end{eqnarray*}
where we have used the assumption that $\frac{\partial u}{\partial x_n}>0>\frac{\partial v}{\partial x_n}$.

Since $u$ and $v$ are both monotone in the $x_n$-direction, and  bounded between $a$ and $b$, we can  integrate $I_1(t)$ with respect to $t \in (-\infty, +\infty)$ to obtain 
\begin{eqnarray}\label{eqI1}
\begin{aligned}
&\int_{-\infty}^{+\infty} I_1(t)\, dt \\
=& R^{-2s} \int_{-\infty}^{+\infty}\int_{t-R}^{t+R}\int_{[-R, R]^{n-1}}\left| \frac{\partial u}{\partial x_n}(x, R)\cdot \frac{\partial u}{\partial y}(x, R) + \frac{\partial v}{\partial x_n}(x, R) \cdot\frac{\partial v}{\partial y}(x,R)\right|\, dx'dx_ndt\\
\leq & C R^{-2s}\int_{-\infty}^{+\infty}\int_{t-R}^{t+R}\int_{[-R, R]^{n-1}} \Big\{ \frac{\partial u}{\partial x_n}(x, R)- \frac{\partial v}{\partial x_n}(x, R)\Big\} \,dx' dx_ndt\\
=& C R^{-2s}\int_{-\infty}^{+\infty}\int_{x_n-R}^{x_n+R}\int_{[-R, R]^{n-1}} \Big\{ \frac{\partial u}{\partial x_n}(x, R)- \frac{\partial v}{\partial x_n}(x, R)\Big\} \,dx'dt dx_n\\
= & 2C R^{1-2s} \int_{[-R, R]^{n-1}} \int_{-\infty}^{+\infty} \Big\{ \frac{\partial u}{\partial x_n}(x, R)- \frac{\partial v}{\partial x_n}(x, R)\Big\} \,dx_n dx'\\
\leq & C(b-a)R^{1-2s} \int_{[-R, R]^{n-1}}dx'\\
\leq & CR^{n-2s}
\end{aligned}
\end{eqnarray}
for any $0<s<1.$

\medskip
\noindent
\textbf{Estimate of $I_2$.} Define
$$
h(y)=\sup_{x\in \mathbb{R}^n} \Big(y^{1-2s} (|\nabla u(x, y)|+|\nabla v(x, y)|) \Big),\quad y>0. 
$$
By \cite[Remark 1.10]{CC2014CVPDE}, 
$$
h(y)\leq 
\begin{cases}
    Cy^{1-2s}, & y\in(0, 1),\\
    Cy^{-2s}, & y\in[1, \infty).
\end{cases}
$$
Since $\frac{\partial u}{\partial x_n}>0>\frac{\partial v}{\partial x_n}$, and $u$ and $v$ are both bounded between $a$ and $b$,  we have  
\begin{eqnarray*}
\begin{aligned}
I_2(t)\leq & \int_{0}^{R}\int_{[-R, R]^{n-1}} h(y)\Big\{\frac{\partial u}{\partial x_n}(x', t-R, y) - \frac{\partial v}{\partial x_n}(x', t-R, y) \Big\} \,dx'dy\\
&+ \int_{0}^{R}\int_{[-R, R]^{n-1}} h(y)\Big\{\frac{\partial u}{\partial x_n}(x', t+R, y) - \frac{\partial v}{\partial x_n}(x', t+R, y) \Big\} \,dx'dy\\
:=& I_{21}(t)+I_{22}(t). 
 \end{aligned}
\end{eqnarray*}
Integrating $ I_{21}(t)$ over $(-\infty, +\infty)$ with respect to $t$ yields
\begin{eqnarray*}
\begin{aligned}
\int_{-\infty}^{+\infty} I_{21}(t) \,dt
= & \int_{-\infty}^{+\infty} \int_{0}^{R}\int_{[-R, R]^{n-1}} h(y)\Big\{\frac{\partial u}{\partial x_n}(x', t-R, y) -  \frac{\partial v}{\partial x_n}(x', t-R, y) \Big\} \,dx'dy dt\\
=& 2(b-a)  \int_{0}^{1}\int_{[-R, R]^{n-1}} h(y) dx'dy+ 2(b-a)  \int_{1}^{R}\int_{[-R, R]^{n-1}} h(y) dx'dy\\
\leq & CR^{n-1} +CR^{n-1}\int_1^R y^{-2s} dy\\
\lesssim & 
\begin{cases}
R^{n-2s}, &  \mbox{if}\quad 0<s<\frac{1}{2},\\
R^{n-1} \ln R, & \mbox{if}\quad s=\frac{1}{2},\\
R^{n-1}, &  \mbox{if}\quad \frac{1}{2}<s<1.
\end{cases}
 \end{aligned}
\end{eqnarray*}
By an argument similar to that for  $I_{21}(t)$, we  obtain 
\begin{eqnarray*}
\int_{-\infty}^{+\infty} I_{22}(t) \,dt \lesssim 
\begin{cases}
R^{n-2s}, &  \mbox{if}\quad 0<s<\frac{1}{2},\\
R^{n-1} \ln R, & \mbox{if}\quad s=\frac{1}{2},\\
R^{n-1}, &  \mbox{if}\quad \frac{1}{2}<s<1. 
\end{cases}
\end{eqnarray*}
As a consequence, we have 
\begin{eqnarray}\label{eqI2}
\int_{-\infty}^{+\infty} I_{2}(t) \,dt \lesssim 
\begin{cases}
R^{n-2s}, &  \mbox{if}\quad 0<s<\frac{1}{2},\\
R^{n-1} \ln R, & \mbox{if}\quad s=\frac{1}{2},\\
R^{n-1}, &  \mbox{if}\quad \frac{1}{2}<s<1. 
\end{cases}
\end{eqnarray}

\medskip
\noindent
\textbf{Estimate of $I_3$.} 
By the monotonicity condition  $\frac{\partial u}{\partial x_n}>0>\frac{\partial v}{\partial x_n}$, 
and the fact that  $u$ and $v$ are bounded between  $a$ and $b$,
 we have
\begin{eqnarray}\label{eqI3}
\begin{aligned}
&\int_{-\infty}^{+\infty} I_3(t)\, dt \\
\leq & \int_{-\infty}^{+\infty}   \int_{0}^{R}\int_{\partial [-R, R]^{n-1}}\int_{t-R}^{t+R} h(y)\Big\{\frac{\partial u}{\partial x_n}(x, y) - \frac{\partial v}{\partial x_n}(x, y) \Big\} \,dx_n d \mathcal{H}^{n-2}(\partial [-R, R]^{n-1})dydt\\
\leq &\int_{-\infty}^{+\infty}   \int_{0}^{R}\int_{\partial [-R, R]^{n-1}}\int_{x_n-R}^{x_n+R} h(y)\Big\{\frac{\partial u}{\partial x_n}(x, y) - \frac{\partial v}{\partial x_n}(x, y) \Big\} \,dt d \mathcal{H}^{n-2}(\partial [-R, R]^{n-1})dydx_n\\
\leq &2R\int_{-\infty}^{+\infty}   \int_{0}^{R}\int_{\partial [-R, R]^{n-1}} h(y)\Big\{\frac{\partial u}{\partial x_n}(x, y) - \frac{\partial v}{\partial x_n}(x, y) \Big\} \, d \mathcal{H}^{n-2}(\partial [-R, R]^{n-1})dydx_n\\
\leq &4R(b-a) \int_{0}^{R}\int_{\partial [-R, R]^{n-1}} h(y) \, d \mathcal{H}^{n-2}(\partial [-R, R]^{n-1})dydx_n\\
\leq &CR^{n-1}\int_0^R h(y) dy\\
 \lesssim &
\begin{cases}
R^{n-2s}, &  \mbox{if}\quad 0<s<\frac{1}{2},\\
R^{n-1} \ln R, & \mbox{if}\quad s=\frac{1}{2},\\
R^{n-1}, &  \mbox{if}\quad \frac{1}{2}<s<1.
 \end{cases}
 \end{aligned}
\end{eqnarray}
Combining \eqref{eqI1}, \eqref{eqI2}, and \eqref{eqI3}, we obtain \eqref{eq4.1}. This completes the proof of Lemma \ref{lem. evolution of E(t)}.
\end{proof}

\section{The limiting profile}

In this subsection, we study the limiting profile obtained by translating the solution along the $\vec{e}_n$-direction toward infinity. Since $(u,v)$ are monotone in $\vec{e}_n$, the limiting profile is well-defined pointwise, as described below.

\begin{definition}\label{def. limiting profile}
We define $\overline{u}, \overline{v} : \mathbb{R}^n_{+} \to \mathbb{R}$ by
\[
    \overline{u}(x',y) := \lim_{x_n \to +\infty} u(x',x_n,y), \quad 
    \overline{v}(x',y) := \lim_{x_n \to +\infty} v(x',x_n,y).
\]
Similarly, we define $\underline{u}, \underline{v} : \mathbb{R}^n_{+} \to \mathbb{R}$ by
\[
    \underline{u}(x',y) := \lim_{x_n \to -\infty} u(x',x_n,y), \quad 
    \underline{v}(x',y) := \lim_{x_n \to -\infty} v(x',x_n,y).
\]
\end{definition}

\begin{lemma}
The pairs $(\overline{u}, \overline{v})$ and $(\underline{u}, \underline{v})$ satisfy the following properties:
\begin{itemize}
    \item[(1)] $(\overline{u}, \overline{v})$ and $(\underline{u}, \underline{v})$ possess globally bounded derivatives of all orders. In particular, $\overline{u}, \overline{v}, \underline{u},$ and $\underline{v}$ are bounded between $a$ and $b$.
    \item[(2)] Both pairs $(\overline{u}, \overline{v})$ and $(\underline{u}, \underline{v})$ satisfy the system \eqref{eq. simplified extension equation} in $\mathbb{R}^{n-1}$.
\end{itemize}
\end{lemma}

\begin{proof}
Since $u$ and $v$ are monotone in the $\vec{e}_n$-direction and bounded between $a$ and $b$ (by Lemma~\ref{lem. range}), the limits $(\overline{u}, \overline{v})$ and $(\underline{u}, \underline{v})$ are well-defined and globally bounded. Moreover, as all derivatives of $(u,v)$ are uniformly bounded in $\mathbb{R}^n$, their limits inherit these derivative bounds.

Indeed, for each $t \in \mathbb{R}$, consider the translated pair
\begin{equation}\label{eq. translated pair}
    \big(u^{(t)}(x), v^{(t)}(x)\big) := \big(u(x',x_n + t,y),\, v(x',x_n + t,y)\big).
\end{equation}
By virtue of the Arzela-Ascoli theorem, there exists a subsequence of $\{(u^{(t)}, v^{(t)})\}_{t \in \mathbb{R}}$ that converges strongly in every $C^{k,\alpha}(\mathbb{R}^n)$ space to the canonical extensions of $(\overline{u}, \overline{v})$ and $(\underline{u}, \underline{v})$ from $\mathbb{R}^{n-1}$ to $\mathbb{R}^n$.
Since each translated pair $(u^{(t)}, v^{(t)})$ satisfies \eqref{eq. simplified extension equation}, the limiting profiles $(\overline{u}, \overline{v})$ and $(\underline{u}, \underline{v})$ also satisfy the same system.
\end{proof}

Next, we show that the limiting profiles $(\overline{u}, \overline{v})$ and $(\underline{u}, \underline{v})$ are one-dimensional. To this end, it suffices to establish the existence of functions $\varphi > 0 > \psi$ such that 
\[
    L\begin{bmatrix}\varphi \\ \psi\end{bmatrix} \in \overline{\mathbb{R}_{+}} \times \overline{\mathbb{R}_{-}}.
\]

\begin{lemma}\label{lem. classification of one limiting profile}
Assume that one of the following holds:
\begin{itemize}
    \item Case 1: $\frac{1}{2} \le s < 1$ and $n = 3$;
    \item Case 2: $0 < s < \frac{1}{2}$ and $n = 2$.
\end{itemize}
For $(\overline{u}, \overline{v})$ as in Definition~\ref{def. limiting profile}, there exist functions $\varphi > 0 > \psi$ defined in $\mathbb{R}^n_{+}$ such that
\[
    \begin{cases}
        \displaystyle \lim_{y \to 0} \big(y^{1-2s} \varphi_y\big) 
        \le W_{uu}(\overline{u}, \overline{v})\,\varphi + W_{uv}(\overline{u}, \overline{v})\,\psi, \\[0.8em]
        \displaystyle \lim_{y \to 0} \big(y^{1-2s} \psi_y\big) 
        \ge W_{vu}(\overline{u}, \overline{v})\,\varphi + W_{vv}(\overline{u}, \overline{v})\,\psi.
    \end{cases}
\]
Moreover, $(\overline{u}, \overline{v})$ is one-dimensional. Specifically, one of the following holds:
\begin{itemize}
    \item[(i)] $(\overline{u}, \overline{v})$ is constant, equal to one of
    \[
        (\overline{u}, \overline{v}) \equiv (a,b), \quad (b,a), \quad \text{or} \quad \left(\sqrt{\frac{1+\omega}{2+\alpha}}, \sqrt{\frac{1+\omega}{2+\alpha}}\right).
    \]
    \item[(ii)] $(\overline{u}, \overline{v})$ depends only on $(x_{n-1}, y)$ (up to a rotation in $\mathbb{R}^{n-1}$), and
    \[
        \frac{\partial \overline{u}}{\partial x_{n-1}} > 0 > \frac{\partial \overline{v}}{\partial x_{n-1}}.
    \]
\end{itemize}
\end{lemma}

\begin{proof}
    We consider a family of functionals related to the translated pair $(u^{(t)},v^{(t)})$ in \eqref{eq. translated pair}. We also denote
    \begin{equation*}
        W^{(t)}=W(u^{(t)},v^{(t)}).
    \end{equation*}
    Let $(\xi,\eta)$ be a pair of smooth functions defined in $\overline{\mathbb{R}^{n+1}_{+}}$ with compact support, and let
    \begin{equation}\label{eq. E_t xi eta s<1}
        \mathcal{E}_{t}(\xi,\eta):=\int_{\mathbb{R}^{n+1}_{+}}y^{1-2s}(|\nabla\xi|^{2}+|\nabla\eta|^{2})dxdy+\int_{\mathbb{R}^{n}}\begin{bmatrix}\xi\\\eta\end{bmatrix}^{T}D^{2}W^{(t)}\begin{bmatrix}\xi\\\eta\end{bmatrix}dx.
    \end{equation}

\textbf{Step 1: Positivity of $\mathcal{E}_{t}(\xi,\eta)$.} 
We first prove that $\mathcal{E}_{t}(\xi,\eta)$ is nonnegative. 
Recall that, by the monotonicity assumption, 
\[
    \frac{\partial u^{(t)}}{\partial x_{n}} > 0 > \frac{\partial v^{(t)}}{\partial x_{n}},
\]
and that the pair 
\(
\big(\tfrac{\partial u^{(t)}}{\partial x_{n}}, \tfrac{\partial v^{(t)}}{\partial x_{n}}\big)
\)
satisfies the Schr\"odinger-type system:
\[
    \begin{cases}
        -\displaystyle\lim_{y\to0}\big(y^{1-2s}\partial_{y}\tfrac{\partial u^{(t)}}{\partial x_{n}}\big)
        + W_{uu}^{(t)}\tfrac{\partial u^{(t)}}{\partial x_{n}} 
        + W_{uv}^{(t)}\tfrac{\partial v^{(t)}}{\partial x_{n}} = 0, \\[0.8em]
        -\displaystyle\lim_{y\to0}\big(y^{1-2s}\partial_{y}\tfrac{\partial v^{(t)}}{\partial x_{n}}\big)
        + W_{vu}^{(t)}\tfrac{\partial u^{(t)}}{\partial x_{n}} 
        + W_{vv}^{(t)}\tfrac{\partial v^{(t)}}{\partial x_{n}} = 0,
    \end{cases}
\]
as established in Lemma~\ref{lem. derivative satisfies the Schrodinger equation}.  

We now set
\[
    (\varphi^{(t)}, \psi^{(t)}) := \Big(\tfrac{\partial u^{(t)}}{\partial x_{n}}, \tfrac{\partial v^{(t)}}{\partial x_{n}}\Big),
    \qquad
    (\sigma, \tau) := \Big(\tfrac{\xi}{\varphi^{(t)}}, \tfrac{\eta}{\psi^{(t)}}\Big).
\]
Then $(\sigma, \tau)$ is a smooth pair with compact support in $\overline{\mathbb{R}^{n+1}_{+}}$.  
With this notation at hand, an integration by parts combined with the equations satisfied by $(\varphi^{(t)}, \psi^{(t)})$ yields 
 \begin{align*}
        \mathcal{E}_{t}(\xi,\eta)
        =&\int_{\mathbb{R}^{n+1}_{+}}y^{1-2s}\Big\{|\varphi^{(t)}\nabla\sigma|^{2}+\nabla\varphi^{(t)}\cdot\nabla(\sigma^{2}\varphi^{(t)})+|\psi^{(t)}\nabla\tau|^{2}+\nabla\psi^{(t)}\cdot\nabla(\tau^{2}\psi^{(t)})\Big\}dxdy\\
        &+\int_{\mathbb{R}^{n}}\Big\{W_{uu}^{(t)}\Big(\varphi^{(t)}\sigma\Big)^{2}+W_{vv}^{(t)}\Big(\psi^{(t)}\tau\Big)^{2}+2W_{uv}^{(t)}\Big(\varphi^{(t)}\sigma\Big)\Big(\psi^{(t)}\tau\Big)\Big\}dx\\
        =&\int_{\mathbb{R}^{n+1}_{+}}y^{1-2s}\Big\{|\varphi^{(t)}\nabla\sigma|^{2}+|\psi^{(t)}\nabla\tau|^{2}\Big\}dxdy\\
        &-\int_{\mathbb{R}^{n+1}_{+}}\Big\{\sigma^{2}\varphi^{(t)}\partial_y(y^{1-2s}\nabla\varphi^{(t)})+\tau^{2}\psi^{(t)}\partial_y(y^{1-2s}\nabla\psi^{(t)})\Big\}dxdy\\
        &-\int_{\mathbb{R}^{n}}\Big\{\sigma^{2}\varphi^{(t)}\lim_{y\to0}(y^{1-2s}\nabla\varphi^{(t)})+\tau^{2}\psi^{(t)}\lim_{y\to0}(y^{1-2s}\nabla\psi^{(t)})\Big\}dx\\
        &+\int_{\mathbb{R}^{n}}\Big\{W_{uu}^{(t)}\Big(\varphi^{(t)}\sigma\Big)^{2}+W_{vv}^{(t)}\Big(\psi^{(t)}\tau\Big)^{2}+2W_{uv}^{(t)}\Big(\varphi^{(t)}\sigma\Big)\Big(\psi^{(t)}\tau\Big)\Big\}dx\\
        =&\int_{\mathbb{R}^{n+1}_{+}}y^{1-2s}\Big\{|\varphi^{(t)}\nabla\sigma|^{2}+|\psi^{(t)}\nabla\tau|^{2}\Big\}dxdy-\int_{\mathbb{R}^{n}}(\sigma\tau)^{2}W_{uv}^{(t)}\varphi^{(t)}\psi^{(t)}dx\geq0.
    \end{align*}
where in the last inequality we have used the positivity property in \eqref{eq. W_uv>0}.

\medskip

\textbf{Step 2: Passing to the limit $t \to +\infty$.}  
We now define $\mathcal{E}_{+\infty}(\xi,\eta)$ analogously to \eqref{eq. E_t xi eta s<1}, with $D^{2}W^{(t)}$ replaced by
\[
    D^{2}W^{(+\infty)} := D^{2}W(u^{(+\infty)}, v^{(+\infty)}),
\]
where
\[
    u^{(+\infty)}(x',x_n,y) := \overline{u}(x',y), 
    \qquad
    v^{(+\infty)}(x',x_n,y) := \overline{v}(x',y).
\]
Since we have already shown that $\mathcal{E}_{t}(\xi,\eta) \ge 0$, we may pass to the limit to conclude that 
\[
    \mathcal{E}_{+\infty}(\xi,\eta) \ge 0
\]
for every compactly supported smooth pair $(\xi,\eta)$ defined on $\overline{\mathbb{R}^{n+1}_{+}}$.  

Next, we define
\[
    D^{2}\overline{W} := D^{2}W(\overline{u}, \overline{v}),
\]
which is now viewed as a matrix-valued function on $\mathbb{R}^{n-1}$.  
For any compactly supported pair $(\xi, \eta)$ in $\mathbb{R}^{n-1} \times \mathbb{R}_{+}$, we introduce
\[
    \overline{\mathcal{E}}(\xi,\eta)
    := \int_{\mathbb{R}^{n-1}\times\mathbb{R}_{+}} 
        y^{1-2s} \big(|\nabla\xi|^{2} + |\nabla\eta|^{2}\big) \, dx'\,dy
       + \int_{\mathbb{R}^{n-1}}
        \begin{bmatrix}\xi \\ \eta\end{bmatrix}^{\!T}
        D^{2}\overline{W}
        \begin{bmatrix}\xi \\ \eta\end{bmatrix} \, dx'.
\]

\medskip
\textbf{Step 3: Positivity of $\overline{\mathcal{E}}(\xi,\eta)$.} 
We now show that $\overline{\mathcal{E}}(\xi,\eta)\geq0$ for any compactly supported pair $(\xi,\eta)$ in $\mathbb{R}^{n-1}\times\mathbb{R}_{+}$. Suppose, on the contrary, that there exists a compactly supported smooth pair $(\xi,\eta)=\big(\xi(x',y),\eta(x',y)\big)$ such that $\overline{\mathcal{E}}(\xi,\eta)<0$. Denote
\begin{equation}\label{eq. -T and diameter}
    \overline{\mathcal{E}}(\xi,\eta)=-\mathcal{T}<0,\qquad 
    \mathcal{S}:=\big(\operatorname{supp}(\xi)\cup \operatorname{supp}(\eta)\big)\Subset\mathbb{R}^{n-1},\qquad 
    \rho:=\sup_{x'\in\mathcal{S}}|x'|.
\end{equation}
Let $Z_{R}(x_{n})\in C^{\infty}(\mathbb{R})$ be a cut-off function satisfying
\begin{equation*}
    0\leq Z_{R}(x_{n})\leq1,\quad 
    Z_{R}(x_{n})\equiv1 \text{ for } |x_{n}|\leq R,\quad 
    Z_{R}(x_{n})\equiv0 \text{ for } |x_{n}|\geq R+1,\quad 
    |\nabla Z_{R}|\leq C.
\end{equation*}
Define
\begin{equation*}
    \big(\xi_{R}(x,y),\eta_{R}(x,y)\big):=\big(\xi(x',y)Z_{R}(x_{n}),\,\eta(x',y)Z_{R}(x_{n})\big),
\end{equation*}
which is a compactly supported pair in $\mathbb{R}^{n+1}_{+}$. Since 
$$\mathcal{E}_{+\infty}(\xi_{R},\eta_{R})\geq0 \,\, \mbox{for all} \,\, R\geq1,$$
we can estimate, for sufficiently large $R$,
\begin{align*}
    \mathcal{E}_{+\infty}(\xi_{R},\eta_{R})
    =&\int_{-R}^{R}\Bigg\{
        \int_{\mathbb{R}^{n-1}\times\mathbb{R}_{+}}
            y^{1-2s}\big(|\nabla\xi|^{2}+|\nabla\eta|^{2}\big)\,dx'dy
        +\int_{\mathbb{R}^{n-1}}
            \begin{bmatrix}\xi\\[2pt]\eta\end{bmatrix}^{T}
            D^{2}\overline{W}
            \begin{bmatrix}\xi\\[2pt]\eta\end{bmatrix}
        dx'
    \Bigg\}dx_{n}\\
    &+\Bigg(\int_{-R-1}^{-R}+\int_{R}^{R+1}\Bigg)
        \int_{\mathbb{R}^{n-1}\times\mathbb{R}_{+}}
            y^{1-2s}\big(|\nabla\xi_{R}|^{2}+|\nabla\eta_{R}|^{2}\big)\,dx'dx_{n}dy\\
    &+\Bigg(\int_{-R-1}^{-R}+\int_{R}^{R+1}\Bigg)
        \int_{\mathbb{R}^{n-1}}
            \begin{bmatrix}\xi_{R}\\[2pt]\eta_{R}\end{bmatrix}^{T}
            D^{2}\overline{W}
            \begin{bmatrix}\xi_{R}\\[2pt]\eta_{R}\end{bmatrix}
        dx'dx_{n}\\
    \leq& -2R\mathcal{T}+C(\xi,\eta)<0,
\end{align*}
which contradicts $\mathcal{E}_{+\infty}(\xi_{R},\eta_{R})\geq0$. Here, the constant $C(\xi,\eta)$ depends only on $\|\xi\|_{C^{0,1}(\mathbb{R}^{n-1}\times\mathbb{R}_{+})}$, $\|\eta\|_{C^{0,1}(\mathbb{R}^{n-1}\times\mathbb{R}_{+})}$, and the radius $\rho$ defined in \eqref{eq. -T and diameter}.

\medskip

    \textbf{Step 4: Construction of $(\varphi,\psi)$.} In this step, we intend to find two entire functions $\varphi>0>\psi$ defined on $\mathbb{R}^{n-1}\times\mathbb{R}_{+}$, still of class $\mathcal{C}$, such that the following Schr\"odinger inequality holds:
    \begin{equation}\label{eq. varphi and psi Schrodinger inequality}
        \left\{\begin{aligned}
            &-\lim_{y\to0}(y^{1-2s}\varphi_{y})+\overline{W}_{uu}\varphi+\overline{W}_{uv}\psi\geq0,\\
            &-\lim_{y\to0}(y^{1-2s}\psi_{y})+\overline{W}_{vu}\varphi+\overline{W}_{vv}\psi\leq0.
        \end{aligned}\right.
    \end{equation}

Since $\overline{\mathcal{E}}$ is positive semi-definite for all compactly supported pairs $(\xi,\eta)$, we consider the following eigenvalue problem. For a large radius $R$, we seek minimizers of the functional
\begin{equation}\label{eq. eigenvalue functional}
    \frac{\overline{\mathcal{E}}(\xi,\eta)}{\displaystyle\int_{\mathbb{R}^{n-1}}(|\xi|^{2}+|\eta|^{2})\,dx'},
    \qquad 
    (\xi,\eta)\in C^{\infty}_{0}([-R,R]^{n-1}\times[0,R]),
\end{equation}
whose value is necessarily non-negative. Without loss of generality, we impose the normalization
\begin{equation}\label{eq. eigen normalized}
    \int_{[-R,R]^{n-1}\times[0,R]} y^{1-2s}\big(|\nabla\xi|^{2}+|\nabla\eta|^{2}\big)\,dx'dy = 1.
\end{equation}
By Lemma~\ref{lem. trace}, condition \eqref{eq. eigen normalized} implies that the trace of $(\xi,\eta)$ on $\mathbb{R}^{n-1}$ satisfies
\begin{equation}\label{eq. gagliardo bound}
    \Big\|\xi\Big|_{\mathbb{R}^{n-1}}\Big\|_{H^{s}(\mathbb{R}^{n-1})}
    +\Big\|\eta\Big|_{\mathbb{R}^{n-1}}\Big\|_{H^{s}(\mathbb{R}^{n-1})}
    \leq C.
\end{equation}

If $(\xi_{k}, \eta_{k})$ is a minimizing sequence of the functional \eqref{eq. eigenvalue functional} satisfying in addition the normalization condition \eqref{eq. eigen normalized}, then this sequence converges weakly in $H^{1}([-R,R]^{n-1}\times[0,R],\, y^{1-2s}\,dx'dy)$. Moreover, by applying the Rellich theorem to the traces $(\xi_{k}, \eta_{k})\big|_{\mathbb{R}^{n-1}}$ on $\mathbb{R}^{n-1}$ and using the Gagliardo norm bound \eqref{eq. gagliardo bound}, we deduce that a subsequence of $(\xi_{k}, \eta_{k})\big|_{\mathbb{R}^{n-1}}$ converges strongly in $L^{2}([-R,R]^{n-1})$.

Denote by $(\varphi_{R},\psi_{R})$ the limit of this subsequence, and let $\lambda_{R}$ be the infimum of the functional \eqref{eq. eigenvalue functional}. Then $\lambda_{R}$ decreases with $R$, and hence $\lambda_{R}$ is uniformly bounded for all $R\geq100$. Furthermore,
\begin{equation*}
    \frac{\overline{\mathcal{E}}(\varphi_{R},\psi_{R})}
    {\displaystyle\int_{\mathbb{R}^{n-1}}(|\varphi_{R}|^{2}+|\psi_{R}|^{2})\,dx'}
    = \inf_{(\xi,\eta)\in C^{\infty}_{0}([-R,R]^{n-1}\times[0,R])}
      \frac{\overline{\mathcal{E}}(\xi,\eta)}
      {\displaystyle\int_{\mathbb{R}^{n-1}}(|\xi|^{2}+|\eta|^{2})\,dx'}
    = \lambda_{R}.
\end{equation*}
Since $\overline{W}_{uv}>0$ (see \eqref{eq. W_uv>0}), it is advantageous to replace $(\varphi_{R},\psi_{R})$ with $(|\varphi_{R}|,-|\psi_{R}|)$, which does not increase the quotient \eqref{eq. eigenvalue functional}. Hence, without loss of generality, we may assume that
\begin{equation*}
    \varphi_{R}\geq0\geq\psi_{R}
    \quad\text{in }[-R,R]^{n-1}\times[0,R],
    \qquad
    \varphi_{R}=\psi_{R}=0
    \quad\text{outside }[-R,R]^{n-1}\times[0,R].
\end{equation*}
The Euler-Lagrange equations associated with the minimizers $(\varphi_{R},\psi_{R})$ are
\begin{align}\label{eq. varphi_R and psi_R equation}
    &\left\{
    \begin{aligned}
        &\mathrm{div}(y^{1-2s}\nabla\varphi_{R})=0,\\
        &\mathrm{div}(y^{1-2s}\nabla\psi_{R})=0
    \end{aligned}
    \right.
   \quad  \text{in }[-R,R]^{n-1}\times[0,R],\\[4pt]
    &\left\{
    \begin{aligned}
        &-\lim_{y\to0}(y^{1-2s}\partial_{y}\varphi_{R})
          +\overline{W}_{uu}\varphi_{R}
          +\overline{W}_{uv}\psi_{R}
          =\lambda_{R}\varphi_{R},\\
        &-\lim_{y\to0}(y^{1-2s}\partial_{y}\psi_{R})
          +\overline{W}_{vu}\varphi_{R}
          +\overline{W}_{vv}\psi_{R}
          =\lambda_{R}\psi_{R}
    \end{aligned}
    \right.
    \quad \text{on }[-R,R]^{n-1}.
\end{align}

Since $\varphi_{R}\geq0\geq\psi_{R}$, the difference
\[
    \mathcal{D}_{R}:=\varphi_{R}-\psi_{R}
\]
is nonnegative in $\mathbb{R}^{n-1}\times\mathbb{R}_{+}$ and not identically zero in $[-R,R]^{n-1}\times[0,R]$. Subtracting the equations in \eqref{eq. varphi_R and psi_R equation} yields
\begin{equation*}
    \left\{
    \begin{aligned}
        &\mathrm{div}(y^{1-2s}\nabla\mathcal{D}_{R})=0,\\
        &\Big|\lim_{y\to0}(y^{1-2s}\partial_{y}\mathcal{D}_{R})\Big|\leq C\mathcal{D}_{R},
    \end{aligned}
    \right.
    \qquad 
    C:=100\max_{R\geq100}\lambda_{R}
      +100\|D^{2}\overline{W}\|_{L^{\infty}(\mathbb{R}^{2})}.
\end{equation*}
By the Harnack principle (see \cite{TX11}), $\mathcal{D}_{R}$ must be strictly positive in $B_{R}$. We normalize
\[
    \mathcal{D}_{R}(0)=\varphi_{R}(0)-\psi_{R}(0)=1.
\]
Applying the Harnack inequality to $\mathcal{D}_{R}$ and Schauder estimates to $(\varphi_{R},\psi_{R})$, we obtain the following uniform bounds:  
there exists a constant $C_{\rho}$ depending only on $\rho$, such that for every $\rho\geq100$ and $R\geq2\rho$,
\begin{itemize}
    \item[(1)] $C_{\rho}^{-1}\leq\mathcal{D}_{R}(x)\leq C_{\rho}$ for all $x\in[-\rho,\rho]^{n-1}\times[0,\rho]$;
    \item[(2)] $\|(\varphi_{R},\psi_{R})\|_{C^{\alpha}([-\rho,\rho]^{n-1}\times[0,\rho])}\leq C_{\rho}$;
    \item[(3)] $\|(y^{1-2s}\partial_{y}\varphi_{R},\,y^{1-2s}\partial_{y}\psi_{R})\|_{C^{\alpha}([-\rho,\rho]^{n-1}\times[0,\rho])}\leq C_{\rho}$.
\end{itemize}
In (2)-(3), we have used $\varphi_{R}\geq0\geq\psi_{R}$ to control $\|\varphi_{R}\|_{L^{\infty}}$ and $\|\psi_{R}\|_{L^{\infty}}$ by $\|\mathcal{D}_{R}\|_{L^{\infty}}$, and then applied the Schauder estimates (see \cite[Lemma~4.5]{CS14}).

Passing to the limit $R\to\infty$ via a diagonal argument, there exists a pair $(\varphi,\psi)$ such that, up to a subsequence,
\begin{itemize}
    \item[(1)] $(\varphi_{R},\psi_{R})\to(\varphi,\psi)$ in $C^{\alpha}([-\rho,\rho]^{n-1}\times[0,\rho])$,
    \item[(2)] $(y^{1-2s}\partial_{y}\varphi_{R},y^{1-2s}\partial_{y}\psi_{R})\to(y^{1-2s}\partial_{y}\varphi,y^{1-2s}\partial_{y}\psi)$ in $C^{\alpha}([-\rho,\rho]^{n-1}\times[0,\rho])$.
\end{itemize}
Since $\mathcal{D}_{R}\geq C_{\rho}^{-1}$ in each $B_{\rho}$ for large $R$, we have $\mathcal{D}_{\infty}:=\varphi-\psi>0$. Moreover, $\varphi\geq0\geq\psi$, and $(\varphi,\psi)$ satisfies the limiting problem:
\begin{align*}
    &\left\{
    \begin{aligned}
        &\mathrm{div}(y^{1-2s}\nabla\varphi)=0,\\
        &\mathrm{div}(y^{1-2s}\nabla\psi)=0
    \end{aligned}
    \right.
    \quad \text{in }\mathbb{R}^{n-1}\times\mathbb{R}_{+},\\[3pt]
    &\left\{
    \begin{aligned}
        &-\lim_{y\to0}(y^{1-2s}\partial_{y}\varphi)
          +\overline{W}_{uu}\varphi+\overline{W}_{uv}\psi
          =\lambda_{\infty}\varphi,\\
        &-\lim_{y\to0}(y^{1-2s}\partial_{y}\psi)
          +\overline{W}_{vu}\varphi+\overline{W}_{vv}\psi
          =\lambda_{\infty}\psi
    \end{aligned}
    \right.
    \quad\text{on }\mathbb{R}^{n-1},
\end{align*}
where $\lambda_{\infty}:=\mathop{\lim}\limits_{R\to\infty}\lambda_{R}\geq0$.  
Thus, the inequalities in \eqref{eq. varphi and psi Schrodinger inequality} hold.

Finally, we show that $\varphi$ and $\psi$ are both nowhere zero. Suppose, for contradiction, that $\varphi(x'_{*},y)=0$ for some $(x'_{*},y)\in\mathbb{R}^{n-1}\times\overline{\mathbb{R}_{+}}$. Then necessarily $y=0$, by the strong maximum principle for $\mathrm{div}(y^{1-2s}\nabla\varphi)=0$.  
Using the positivity of $\overline{W}_{uv}$, the negativity of $\psi$, and the boundedness of $\overline{W}_{uu}$, we obtain
\begin{equation*}
    -\lim_{y\to0}(y^{1-2s}\partial_{y}\varphi)
    \geq -\overline{W}_{uu}\varphi
    \geq -\|\overline{W}_{uu}\|_{L^{\infty}(\mathbb{R}^{2})}\varphi.
\end{equation*}
 As we have assumed $\varphi(x'_{*},0)=0$ for some $x'_{*}\in\mathbb{R}^{n-1}$, then it follows from the Harnack principle (see \cite{TX11}) that $\varphi\equiv0$ in $\mathbb{R}^{n-1}\times\mathbb{R}_{+}$. As a result, $\psi(x')\equiv-\mathcal{D}_{\infty}(x')>0$ everywhere in $\mathbb{R}^{2}$, and thus
    \begin{equation*}
        -\lim_{y\to0}(y^{1-2s}\partial_{y}\varphi)\geq-\overline{W}_{uv}\psi>0\mbox{ everywhere in }\mathbb{R}^{n-1}.
    \end{equation*}
This contradicts the earlier conclusion that $\varphi\equiv 0$, and therefore we must have $\varphi>0$ in $\mathbb{R}^{n-1}$. By a similar argument, we also deduce that $\psi<0$ in $\mathbb{R}^{n-1}$.
\medskip

\textbf{Step 5: Classification of the limiting profile.}  
We now describe the properties of the limiting profile $(\overline{u},\overline{v})$. Using the pair $(\varphi,\psi)$ constructed in Step 4, we apply Corollary~\ref{cor. n=2 is more obvious} to $(\overline{u},\overline{v})$.  

From this, we conclude that either $(\overline{u},\overline{v})$ is constant, or it depends only on $(x_{n-1},y)$, up to a rotation in $\mathbb{R}^{n-1}$.  

In the first case, the only constant positive solutions are
\begin{equation*}
    (\overline{u},\overline{v})\equiv(a,b), \quad (b,a), \quad \text{or} \quad \Big(\sqrt{\frac{1+\omega}{2+\alpha}},\,\sqrt{\frac{1+\omega}{2+\alpha}}\Big).
\end{equation*}

In the second case, the monotonicity of $(\overline{u},\overline{v})$ with respect to $x_{n-1}$ also follows directly from Corollary~\ref{cor. n=2 is more obvious}.
\end{proof}

One can naturally expect that Lemma~\ref{lem. classification of one limiting profile} also applies to the other limiting profile $(\underline{u},\underline{v})$. Accordingly, we have the following description of the behavior of both limiting profiles.

\begin{lemma}\label{lem. classification of two limiting profiles}
    Assume that one of the following two cases holds:
    \begin{itemize}
        \item Case A: $\displaystyle\frac{1}{2}\leq s<1$ and $n=3$,
        \item Case B: $\displaystyle0<s<\frac{1}{2}$ and $n=2$.
    \end{itemize}
    Let $(\overline{u},\overline{v})$ and $(\underline{u},\underline{v})$ be the two limiting profiles as defined in Definition~\ref{def. limiting profile}. If $\displaystyle\frac{\partial u}{\partial x_{n}}>0>\frac{\partial v}{\partial x_{n}}$ in $\mathbb{R}^{n+1}_{+}$,
    then at least one of the following occurs:
    \begin{itemize}
        \item Case 1: $(\overline{u},\overline{v})\equiv(b,a)$.
        \item Case 2: $(\underline{u},\underline{v})\equiv(a,b)$.
        \item Case 3: Up to a rotation, $(\overline{u},\overline{v})$ depends only on $(x_{n-1},y)$ and $\displaystyle\frac{\partial\overline{u}}{\partial x_{n-1}}>0>\frac{\partial\overline{v}}{\partial x_{n-1}}$.
        \item Case 4: Up to a rotation, $(\underline{u},\underline{v})$ depends only on $(x_{n-1},y)$ and $\displaystyle\frac{\partial\underline{u}}{\partial x_{n-1}}>0>\frac{\partial\underline{v}}{\partial x_{n-1}}$.
    \end{itemize}
    In Cases 1 and  3, we have that the Ginzburg-Landau energy of the pair $(\overline{u},\overline{v})$ in $\mathbb{R}^{n-1}\times\mathbb{R}_{+}$, as defined in \eqref{eq. Ginzburg-Landau energy, s<1}, satisfies the growth rate estimate below (here, $B_{R}'$ represents the disc in $\mathbb{R}^{2}$ with radius $R$) for all $R\geq100$:
    \begin{equation}\label{eq. limiting profile GL energy growth}
        J(\overline{u},\overline{v},[-R,R]^{n-1}\times[0,R])\leq\left\{\begin{aligned}
            &C R^{1-2s},&\mbox{if }&0<s<\frac{1}{2}\mbox{ and }n=2,\\
            &C R\ln{R},&\mbox{if }&s=\frac{1}{2}\mbox{ and }n=3,\\
            &C R,&\mbox{if }&\frac{1}{2}<s\leq1\mbox{ and }n=3.
        \end{aligned}\right.
    \end{equation}
    
    Similarly, in Cases 2 and 4, the same estimate holds for the pair $(\underline{u},\underline{v})$.
\end{lemma}

\begin{proof}
\textbf{Step 1: Classification of the four cases.} 
We first assume that both Case 3 and Case 4 fail and aim to show that at least one of Case 1 or Case 2 must hold. By Lemma~\ref{lem. classification of one limiting profile}, this implies that both $(\overline{u},\overline{v})$ and $(\underline{u},\underline{v})$ are constant solutions.  

Suppose, in addition, that Case 1 also fails. Then we must have
\begin{equation*}
    (\overline{u},\overline{v}) \equiv \Big(\sqrt{\frac{1+\omega}{2+\alpha}}, \sqrt{\frac{1+\omega}{2+\alpha}}\Big) \quad \text{or} \quad (\overline{u},\overline{v}) \equiv (a,b).
\end{equation*}
Since $\displaystyle \frac{\partial u}{\partial x_n} > 0 > \frac{\partial v}{\partial x_n}$, we deduce that $\underline{u} < \overline{u}$ and $\underline{v} > \overline{v}$. In particular,
\begin{equation*}
    \underline{u} < \max\Big\{\sqrt{\frac{1+\omega}{2+\alpha}}, a\Big\} = \sqrt{\frac{1+\omega}{2+\alpha}}, 
    \quad 
    \overline{v} > \min\Big\{\sqrt{\frac{1+\omega}{2+\alpha}}, b\Big\} = \sqrt{\frac{1+\omega}{2+\alpha}}.
\end{equation*}
The only constant solution remaining for $(\underline{u},\underline{v})$, according to Lemma~\ref{lem. classification of one limiting profile}, is therefore $(a,b)$. This completes the classification of the limiting profiles.

    \textbf{Step 2: Energy estimate.} 
Next, we prove the growth rate estimate \eqref{eq. limiting profile GL energy growth} for the Ginzburg-Landau energy. The estimates in Case 1 and Case 2 are trivial, as both the infinitesimal Dirichlet energy and the potential energy vanish identically. The more challenging cases are Case 3 and Case 4; we focus on Case 3, noting that Case 4 follows by the same argument.

In Case 3, Lemma~\ref{lem. classification of one limiting profile} implies that
\begin{equation*}
    (\overline{u}(x',y), \overline{v}(x',y)) = (U(x_{n-1},y), V(x_{n-1},y))
\end{equation*}
for some one-dimensional extended solution $(U(t,y), V(t,y))$ satisfying the monotonicity
\begin{equation}\label{eq. one-D monotone}
    \partial_t U(t,y) > 0 > \partial_t V(t,y).
\end{equation}

To prove \eqref{eq. limiting profile GL energy growth}, it suffices to reduce the problem to the $x_1$-direction in the case $n=3$ (corresponding to $\frac{1}{2}\leq s<1$) and show that the one-dimensional pair $(U,V)$ satisfies
    \begin{equation*}
        J(U,V,[-R,R]\times[0,R])\leq\left\{\begin{aligned}
            &C R^{1-2s},&\mbox{if }&0<s<\frac{1}{2},\\
            &C\ln{R},&\mbox{if }&s=\frac{1}{2},\\
            &C,&\mbox{if }&\frac{1}{2}<s\leq1.
        \end{aligned}\right.
    \end{equation*}
More precisely, we need to establish  
    \begin{equation}\label{eq. 1 d energy bound}
        \int_{0}^{R}\int_{-R}^{R}\frac{|\nabla U|^{2}+|\nabla V|^{2}}{2}dtdy+\int_{-R}^{R}W(U,V)dt\leq\left\{\begin{aligned}
            &C R^{1-2s},&\mbox{if }&0<s<\frac{1}{2},\\
            &C\ln{R},&\mbox{if }&s=\frac{1}{2},\\
            &C,&\mbox{if }&\frac{1}{2}<s\leq1.
        \end{aligned}\right.
    \end{equation}

The idea is to apply Lemma~\ref{lem. evolution of E(t)}. For the one-dimensional pair $(U,V)$ with monotonicity \eqref{eq. one-D monotone}, we define the limiting profiles $(\overline{U},\overline{V})$ and $(\underline{U},\underline{V})$ as in Definition~\ref{def. limiting profile}. Since these profiles are defined on $\mathbb{R}^{0}\times \mathbb{R}_+ = \{0\} \times \mathbb{R}_+$, they must be constant and satisfy the system \eqref{eq. simplified extension equation} in $\{0\} \times \mathbb{R}_+$. Consequently, the only possible values are
\begin{equation*}
    (\overline{U},\overline{V}), (\underline{U},\underline{V}) \in \Big\{ (a,b), (b,a), \Big(\sqrt{\frac{1+\omega}{2+\alpha}}, \sqrt{\frac{1+\omega}{2+\alpha}}\Big) \Big\}.
\end{equation*}

Following the discussion in Step 1, at least one of the following must hold:
\begin{itemize}
    \item Case (i): $(\overline{U},\overline{V}) = (b,a)$,
    \item Case (ii): $(\underline{U},\underline{V}) = (a,b)$.
\end{itemize}

Without loss of generality, assume Case (i) holds. Let $H$ be sufficiently large and send $H \to \infty$. Then
\begin{equation*}
    \lim_{H\to+\infty} E(H) = 
    \lim_{H\to+\infty} \Bigg\{ \int_0^R \int_{H-R}^{H+R} \frac{|\nabla U|^2 + |\nabla V|^2}{2} \, dt \, dy
    + \int_{H-R}^{H+R} W(U,V) \, dt \Bigg\} = 0.
\end{equation*}
Applying the one-dimensional version of Lemma~\ref{lem. evolution of E(t)} to $(U,V)$ then establishes \eqref{eq. 1 d energy bound}. This completes Step 2 of the proof of Lemma~\ref{lem. classification of two limiting profiles}.
\end{proof}

\section{Proof of Lemma~\ref{lem. energy growth} and Theorem~\ref{thm. main theorem}}
Finally, we prove Lemma~\ref{lem. energy growth} and Theorem~\ref{thm. main theorem}.

\begin{proof}[Proof of Lemma~\ref{lem. energy growth}]
By Lemma~\ref{lem. classification of two limiting profiles}, we may assume without loss of generality that either Case 1 or Case 3 holds. Consequently, the estimate \eqref{eq. limiting profile GL energy growth} is valid.  

Applying Lemma~\ref{lem. evolution of E(t)} in the case $n\leq3$ with $s\geq\frac{1}{2}$, or in the case $n\leq2$ with $s<\frac{1}{2}$, we obtain
\begin{equation}\label{eq. proof of main, energy bound}
J(u,v,[-R,R]^{n}) 
\leq \int_{-\infty}^{\infty} \left| \frac{dE}{dt} \right| dt + R \cdot J(\overline{u},\overline{v},[-R,R]^{2}) 
\leq C R^{2}\ln{R}.
\end{equation}

For sufficiently large $R$, the same estimate holds when the cube $[-R,R]^{n}$ is replaced by the ball $B_{R}$. This completes the proof of Lemma~\ref{lem. energy growth}.
\end{proof}

\begin{proof}[Proof of Theorem~\ref{thm. main theorem}]
We argue similarly as in Corollary~\ref{cor. n=2 is more obvious}.  
We set
\begin{equation*}
    (\varphi,\psi):=\left(\frac{\partial u}{\partial x_{n}},\frac{\partial v}{\partial x_{n}}\right), 
    \qquad
    (\xi_{i},\eta_{i}):=\left(\frac{\partial u}{\partial x_{i}},\frac{\partial v}{\partial x_{i}}\right), 
    \quad i<n.
\end{equation*}
By Lemma~\ref{lem. derivative satisfies the Schrodinger equation} together with Lemma~\ref{lem. quadratic form Q is negative semi-definite}, we deduce that for
\begin{equation*}
    (\sigma_{i},\tau_{i})
    :=\left(\frac{\xi_{i}}{\varphi},\frac{\eta_{i}}{\psi}\right),
\end{equation*}
there exists a strictly positive function $\mathcal{P}(x)$ such that
\begin{equation*}
    Q_{(\varphi,\psi)}
    \begin{bmatrix}
        \sigma_{i}\\
        \tau_{i}
    \end{bmatrix}
    \;\leq\; -(\sigma_{i}-\tau_{i})^{2}\,\mathcal{P}(x)
    \qquad \text{everywhere in }\mathbb{R}^{n}.
\end{equation*}

In view of the growth rate estimate for the Ginzburg-Landau energy established in Lemma~\ref{lem. energy growth}, we can apply Lemma~\ref{lem. Liouville-type result with cutoff function}. This yields that $(\sigma_{i},\tau_{i})$'s are all constants for $i<n$. Moreover, $\sigma_{i}=\tau_{i}$ for $i<n$.  

Finally, by repeating the argument used in Corollary~\ref{cor. n=2 is more obvious}, we conclude that the pair $(u,v)$ is one-dimensional. By rotating the space, we may assume that now $(u,v)$ depends only on $x_n$. Its limiting profiles $(\overline{u},\overline{v})$ and $(\underline{u},\underline{v})$ must then be constant solutions, and $(\overline{u},\overline{v}) = (b,a)$, $(\underline{u},\underline{v}) = (a,b)$. Here, we remark that the limiting profiles cannot be $(\sqrt{\frac{1+\omega}{2+\alpha}},\sqrt{\frac{1+\omega}{2+\alpha}})$, otherwise it violates Step 3 of the proof of Lemma~\ref{lem. classification of one limiting profile}.
\end{proof}

\vspace{2mm}
\noindent \textbf{Acknowledgments.}
Wu is partially supported by National Natural Science Foundation of China (Grant No. 12401133) and the Guangdong Basic and Applied Basic Research Foundation (2025B151502069).

\vspace{2mm}
\noindent \textbf{Conflict of interest.} The authors do not have any possible conflicts of interest.

\vspace{2mm}

\noindent \textbf{Data availability statement.}
 Data sharing is not applicable to this article, as no data sets were generated or analyzed during the current study.




\end{document}